\documentclass[11pt]{article}
\usepackage[singlespacing]{setspace}
\usepackage{amsmath, amssymb, amscd, amsthm, amsfonts,bbm}
\usepackage{mathtools}
\usepackage[mathscr]{euscript}
\usepackage{mathrsfs}
\usepackage{enumitem}
\usepackage{graphicx}
\usepackage{float}
\usepackage{tocloft}
\usepackage{psfrag}
\usepackage{caption}
\usepackage{natbib}
\usepackage{subcaption}
\RequirePackage{doi}
\usepackage{hyperref}
\hypersetup{
colorlinks = true,
urlcolor   = blue,
citecolor  = blue,
linkcolor  = blue,
}
\usepackage{xcolor}
\usepackage{scalerel,stackengine}
\usepackage[colorinlistoftodos]{todonotes}
\usepackage{authblk}
\usepackage{fancyhdr}
\usepackage{academicons}
\definecolor{orcidlogocol}{HTML}{A6CE39}
\usepackage{lineno}

\newlength\FHoffset
\fancyheadoffset{\FHoffset}

\newlength\FHleft
\newlength\FHright

\renewcommand{\headrulewidth}{1.0pt} 
\newbox\FHline
\setbox\FHline=\hbox{\hsize=\paperwidth%
  \hspace*{\FHleft}%
  \rule{\dimexpr\headwidth-\FHleft-\FHright\relax}{\headrulewidth}\hspace*{\FHright}%
}

\title{\textbf{Almost anomalous dissipation in advection-diffusion of a divergence-free passive vector}}
\author[]{\large 
\textbf{Anuj Kumar}\footnote{Department of Mathematics, University of California Berkeley, CA 94720, USA. \textit{Email:}  \href{mailto:anujkumar@berkeley.edu}{anujkumar@berkeley.edu}. 
}
}
\date{}

\newtheoremstyle{mystyle}%                % Name
  {}%                                     % Space above
  {}%                                     % Space below
  {\itshape}%                                     % Body font
  {}%                                     % Indent amount
  {\bfseries}%                            % Theorem head font
  {.}%                                    % Punctuation after theorem head
  { }%                                    % Space after theorem head, ' ', or \newline
  {\thmname{#1}\thmnumber{ #2}\thmnote{ (#3)}}%                                     % Theorem head spec (can be left empty, meaning `normal')

\theoremstyle{mystyle}
\newtheorem{theorem}{Theorem}[section]
\newtheorem{proposition}[theorem]{Proposition}
\newtheorem{lemma}{Lemma}[section]

\theoremstyle{definition}
\newtheorem{remark}{Remark}[section]

\newcommand\norm[1]{\left\lVert#1\right\rVert}

\newcommand{\bs}[1]{\boldsymbol{#1}}

\newcommand{\wt}[1]{\widetilde{#1}}
\newcommand{\ol}[1]{\overline{#1}}

\stackMath
\newcommand\reallywidecheck[1]{%
\savestack{\tmpbox}{\stretchto{%
  \scaleto{%
    \scalerel*[\widthof{\ensuremath{#1}}]{\kern-.6pt\bigwedge\kern-.6pt}%
    {\rule[-\textheight/2]{1ex}{\textheight}}%WIDTH-LIMITED BIG WEDGE
  }{\textheight}% 
}{0.5ex}}%
\stackon[1pt]{#1}{\scalebox{-1}{\tmpbox}}%
}

\makeatletter
\newcommand\RedeclareMathOperator{%
  \@ifstar{\def\rmo@s{m}\rmo@redeclare}{\def\rmo@s{o}\rmo@redeclare}%
}
\newcommand\rmo@redeclare[2]{%
  \begingroup \escapechar\m@ne\xdef\@gtempa{{\string#1}}\endgroup
  \expandafter\@ifundefined\@gtempa
     {\@latex@error{\noexpand#1undefined}\@ehc}%
     \relax
  \expandafter\rmo@declmathop\rmo@s{#1}{#2}}
\newcommand\rmo@declmathop[3]{%
  \DeclareRobustCommand{#2}{\qopname\newmcodes@#1{#3}}%
}
\@onlypreamble\RedeclareMathOperator
\makeatother

 \RedeclareMathOperator{\div}{div}
 
\DeclareMathOperator\supp{supp}

\def\Xint#1{\mathchoice
{\XXint\displaystyle\textstyle{#1}}%
{\XXint\textstyle\scriptstyle{#1}}%
{\XXint\scriptstyle\scriptscriptstyle{#1}}%
{\XXint\scriptscriptstyle\scriptscriptstyle{#1}}%
\!\int}
\def\XXint#1#2#3{{\setbox0=\hbox{$#1{#2#3}{\int}$ }
\vcenter{\hbox{$#2#3$ }}\kern-.6\wd0}}

\def\dashint{\Xint-}

\numberwithin{equation}{section}

\begingroup

\endgroup

\newtheorem{question}[theorem]{Question}

\makeatletter

\begin{document}

\maketitle

\begin{abstract}
We explore the advection-diffusion of a passive vector described by $\partial_t \bs{u} + \bs{U} \cdot \nabla \bs{u} = - \nabla p + \nu \Delta \bs{u}$, where both $\bs{U}$ and $\bs{u}$ are divergence-free velocity fields. We approach this equation from an input/output perspective, with $\bs{U}$ as the input and $\bs{u}$ as the output. This input/output viewpoint has been widely applied in recent studies on passive scalar equation, in the context of anomalous dissipation, mixing, optimal scalar transport, and nonuniqueness problems. What makes the passive vector equation considerably more challenging compared to the passive scalar equation is the lack of a Lagrangian perspective due to the presence of pressure.

In this paper, rather than requiring
 $\bs{U}$ and $\bs{u}$ to be identical (as in the Navier-Stokes equation), we require $\bs{U}$ and $\bs{u}$ to be identical only in certain characteristics. We focus on the case where the characteristics in question are anomalous and enhanced dissipation. We study the advection-diffusion of a passive vector in a Couette flow configuration. The main result of this paper is a construction of the velocity field $\bs{U}$ for which the energy dissipation scales as $(\log \nu^{-1})^{-2}$ such that the energy dissipation in velocity field $\bs{u}$ scales at least as  $(\log \nu^{-1})^{-2}$. This means that both $\bs{U}$ and $\bs{u}$ exhibit near-anomalous dissipation, where the rate of energy dissipation decreases more slowly than any power-law $\nu^{\alpha}$ for any $\alpha > 0$. The result in this paper is not just a mathematical construct; it closely resembles the behavior of turbulent flow in a channel. The $(\log \nu^{-1})^{-2}$ decrease of energy dissipation is predicted by phenomenological theories of wall-bounded turbulence, a prediction that has been extensively validated through  experiments and numerical simulations. Finally, our study motivates the potential to close the gap between $\bs{U}$ and $\bs{u}$—at least with the help of computational tools—and to investigate whether the $(\log \nu^{-1})^{-2}$ decrease in energy dissipation can be maintained.

 %Furthermore, the construction of $\bs{U}$ relies on branching flows with self-similar hierarchical structures, which bear similarities to Townsend’s attached-eddy hypothesis.

 %However, studying these same characteristics is considerably challenging for a passive vector equation because of the absence of a Lagrangian perspective due to the presence of pressure.
\end{abstract}

\section{Introduction}
\label{sec: first problem}
In this paper, we are interested in the advection-diffusion of a passive divergence-free vector field $\bs{u}$. The equations governing the evolution of $\bs{u} : \Omega \to \mathbb{R}^n$ are given by
\noindent
\begin{subequations}
\begin{align}
& \nabla \cdot \bs{u} = 0, 
\label{eqn: 1 prb div free}\\
& \partial_t \bs{u} + \bs{U} \cdot \nabla \bs{u} = - \nabla p + \nu \Delta \bs{u},
\label{eqn: 1 prb momentum}
\end{align}
\label{eqn: conv diff u}
\end{subequations}

\noindent
Here, $\Omega \subseteq \mathbb{R}^d$ is the domain of interest and $d$ is either $2$ or $3$. The velocity field $\bs{U}: \Omega \to \mathbb{R}^d$ (not necessarily the same as $\bs{u}$) is the advecting velocity field and is divergence free as well:
\begin{align}
\nabla \cdot \bs{U} = 0.
\end{align}
 The pressure $p: \Omega \to \mathbb{R}$ is a scalar that ensures  the vector field $\bs{u}$ remains divergence-free. Both velocity fields $\bs{u}$ and $\bs{U}$ satisfy the same initial and boundary conditions.  
 
In this paper, we explore these equations from the perspective of control theory, where the advecting field $\bs{U}$ can be seen as the control and the advected field $\bs{u}$ is the 
output. Of course, when the input $\bs{U}$ and output $\bs{u}$ are exactly the same then we recover the incompressible Navier--Stokes equations:
 \begin{subequations}
\begin{align}
& \nabla \cdot \bs{u} = 0, 
\label{eqn: NSE div free}\\
& \partial_t \bs{u} + \bs{u} \cdot \nabla \bs{u} = - \nabla p + \nu \Delta \bs{u}.
\label{eqn: NSE momentum}
\end{align}
\label{eqn: NSE}
\end{subequations}

\noindent
However, rather than insisting $\bs{U}$ and $\bs{u}$ to be precisely the same, we make the situation a bit more flexible by requiring $\bs{U}$ and $\bs{u}$ to be identical only in certain characteristics. This paper focuses on the scenario where the characteristic in question is anomalous or enhanced dissipation.

We note that this input/output perspective has been central to many research studies in fluid dynamics in recent years, particularly concerning the scalar equation (with or without diffusion)
\begin{align}
\partial_t \theta + \bs{U} \cdot \nabla \theta = \kappa \Delta \theta, \quad \kappa \geq 0,
\label{eqn: scalar diff}
\end{align}
where the overarching question always is whether it is possible to design a vector field $\bs{U}$ (subject to certain constraints) such that the scalar $\theta$ exhibits desired characteristics. This input/output point-of-view has been successfully utilized in various contexts, including studies on anomalous dissipation in passive scalars \citep{drivas2022anomalous, colombo2023anomalous, armstrong2023anomalous, elgindi2023norm}, optimal mixing of scalars \citep{lin2011optimal, iyer2014lower, YaoZlatos17, AlbertiCrippaMazzucato19, ElgindiZlatosuniversalmixer}, optimal scalar transport \citep{hassanzadeh2014wall, doering2019optimal, kumar2022three, tobasco2021optimal}, and even problems related to the nonuniqueness of solutions to the transport equation \eqref{eqn: scalar diff} \citep{depauw2003non, kumar2023nonuniqueness, brue2024sharp}. 

Many of the studies referenced above have used the Lagrangian perspective available for the transport equation in an essential way. However, we note that analyzing equations (\ref{eqn: conv diff u}{\color{blue}a-b}) is much more challenging, as the Lagrangian viewpoint is no longer applicable in general. As such, it is not at all clear that creating small scales in $\bs{U}$ will lead to small scale creation in $\bs{u}$. 

In this paper, we study equations (\ref{eqn: conv diff u}{\color{blue}a-b}) within a Couette flow configuration (see Figure \ref{fig: Couette}). One of the most important highlight of our work is the construction of a family of velocity fields $\{\bs{U}_\nu\}$ with almost anomalous dissipation, in particular the energy dissipation scales as ($(\log \nu^{-1})^{-2}$), such that the corresponding solutions  $\{\bs{u}_\nu\}$ also has energy dissipation scaling at least as ($(\log \nu^{-1})^{-2}$). Notably, the rate of energy dissipation in both $\{\bs{U}_\nu\}$ and $\{\bs{u}_\nu\}$ is slower than power law $\nu^\alpha$ for any $\alpha > 0$. Our result is sharp in the sense that a logarithmic decrease of energy dissipation in a channel flow is a prediction of phenomenological theories of wall turbulence \citep{schlichting2016boundary} and have been observed in experiments \citep{smits2011high, marusic2013logarithmic} to an enormous degree. As a side note, we observe that anomalous dissipation in channel or pipe flow geometries should, in principle, be achievable when there is sufficient transversal flow toward the boundary. For instance, see the work of \cite{nguyen2011energy} on  dipole crashing into a wall in two dimensions.

\subsection{Purpose of studying equation (\ref{eqn: conv diff u}) and anomalous dissipation in a vector field}
The main purpose of introducing equations (\ref{eqn: conv diff u}{\color{blue}a-b}) in this paper is to discuss a problem pertaining to anomalous dissipation in a divergence-free velocity field. However, the significance of these equations extends far beyond this topic. For instance, questions related to nonuniqueness and loss of regularity in the Navier–Stokes equations are among the most pressing concerns in mathematical fluid dynamics. Equations (\ref{eqn: conv diff u}{\color{blue}a-b}) provide a valuable platform for studying these problems without the complications introduced by the nonlinear term in the Navier–Stokes equations. At the same time, these equations still require us to confront the true nature of the pressure term, which is absent in the advection-diffusion of a scalar. Consequently, studying equations (\ref{eqn: conv diff u}) can allow us to gain deeper insights into the role of pressure in incompressible fluid flows.

An important characteristics of turbulent flow is that the rate of energy dissipation $\varepsilon$ becomes independent of viscosity in the limit of vanishingly small viscosity \citep{sreenivasan1998update, frisch1995turbulence}. Mathematically, it means that the solution of the Navier--Stokes equations (\ref{eqn: NSE}) obey
\begin{align}
\varepsilon \coloneqq \nu \langle |\nabla \bs{u}|^2 \rangle \geq c > 0, \quad \forall \nu > 0,
\label{prb: anom dissp exp}
\end{align}
for some constant $c$. Here, the angle brackets denote the long-time volume average:
\begin{align}
\langle [\, \cdot \,] \rangle \coloneqq \limsup_{T \to \infty} \frac{1}{T} \int_{0}^{T}  \dashint_{\Omega} [\, \cdot \,] \, {\rm d} \bs{x} \, {\rm d} t, \qquad \text{where} \qquad \dashint_{\Omega} [\, \cdot \,] \, {\rm d} \bs{x} \coloneqq \frac{1}{|\Omega|} \int_{\Omega} [\, \cdot \,] \, {\rm d} \bs{x}.
\label{prb: long time avg}
\end{align}
This phenomenon, known as the anomalous dissipation of energy and sometimes also referred to as the ``zeroth law'' of turbulence due to its fundamental nature \citep{sreenivasan1984scaling}. This phenomenon has been extensively validated through numerous experiments and direct numerical simulations \citep{pearson2002measurements, kaneda2003energy}. 
%For example, figure XX shows a plot of the drag coefficient as a function of the Reynolds number.

The phenomenon of anomalous dissipation can initially seem counterintuitive: how can a system continue to dissipate energy at a constant rate if the factor responsible for friction (in our case it would be viscosity) is gradually eliminated?  For instance, in a rigid body system with finite number of degree of freedoms, such as a double pendulum, the rate of energy dissipation decreases linearly with the friction factor and eventually approaches zero. However, this analogy is not applicable for fluid systems, where the number of degrees of freedom can increase without bound as viscosity approaches zero. The turbulent flows acquire these increase number of degree of freedom through the increase in the number of small scales in the flow. The increase in small scales is responsible for the growth of $\langle |\nabla \bs{u}|^2 \rangle$ and for turbulent flows this quantity grows precisely as $\nu^{-1}$ leading to anomalous dissipation of the energy.

Assuming that the turbulent flows acquires many of its mean properties through invariant solutions (steady or time-periodic) to the Navier–Stokes equations \citep{kawahara2012significance}, one would expect to find a family of these simple invariant solutions that also exhibit anomalous dissipation. This raises the following question:
\begin{question}[Solution of NSE with anomalous dissipation]
Is there a family of solution $\{\bs{u}_\nu\}_{\nu > 0}$ to the Navier--Stokes equations (\ref{eqn: NSE}) for which 
\begin{align}
\nu \langle |\nabla \bs{u}_{\nu}|^2 \rangle \geq c > 0 \quad \text{for all} \quad \nu > 0?
\label{prb: anom dissp NSE}
\end{align}
\label{Q: anom dissp NSE}
\end{question}
Interestingly, to date, there are no known examples—whether steady or time-dependent—flows that follow the relation (\ref{prb: anom dissp NSE}) \citep{drivas2022self}. The two physically motivated settings to consider this question are the following:
\begin{enumerate}[label = (\roman*)]
\item Tangential velocity-driven flow: In this setting $\Omega$ is a smooth domain with impermeable boundaries, where the flow is driven by a prescribed tangential boundary condition (independent of viscosity). 
\item Smooth body forcing-driven flow: The flow inside a periodic domain or a smooth bounded domain with homogeneous Dirichlet boundary condition  driven by the presence of a smooth forcing $\bs{f}$ (independent of viscosity) in the equation (\ref{eqn: conv diff u}{\color{blue} b}). In this context, the quantity that is expected to become independent of viscosity is the energy dissipation per unit kinetic energy, i.e., $\frac{\nu \langle |\nabla \bs{u}_{\nu}|^2 \rangle}{\langle | \bs{u}_{\nu}|^2 \rangle}$.
\end{enumerate}

We note that recently \citet{brue2023anomalous} considered the case where the forcing is allowed to depend on $\nu$. Constructing velocity fields that answers the Question \ref{Q: anom dissp NSE} in the two settings mentioned above at present remains out of reach. Inspired by this question, one can pose a similar question within the framework of passive vector equation (\ref{eqn: conv diff u}). 
\begin{question}[Anomalous dissipation in a divergence-free passive vector field]
For some $\nu_0 > 0$, can one construct a family of smooth divergence-free vector fields $\{\bs{U}_\nu\}_{\nu_0 > \nu > 0}$ exhibiting anomalous dissipation, in particular, obeying
\begin{align}
\lim_{\nu \to 0} \nu \langle |\nabla \bs{U}_{\nu}|^2 \rangle = c_1 > 0
\end{align}
such that the corresponding family of solutions $\{\bs{u}_\nu\}_{\nu_0 > \nu > 0}$ to (\ref{eqn: conv diff u}) satisfies
\begin{align}
\nu \langle |\nabla \bs{u}_{\nu}|^2 \rangle \geq c_2 > 0 \quad \text{for all} \quad \nu \in (0, \nu_0).
\end{align}
\label{Q: anom dissp div free}
\end{question}
%In simple words, this question is asking for a advecting velocity field $\{\bs{U}_\nu\}_{\nu > 0}$ that exhibit anomalous dissipation, such that the velocity field it transports, $\{\bs{u}_\nu\}_{\nu > 0}$, also exhibit anomalous dissipation. 

At first glance, it might appear that answering this question shouldn't be too difficult:  if the advecting velocity field $\{\bs{U}_\nu\}_{\nu > 0}$ exhibits anomalous dissipation, then it is plausible that the velocity field it transports, $\{\bs{u}_\nu\}_{\nu > 0}$, should also display anomalous dissipation. For examples, if one creates small scale in the velocity field $\{\bs{U}_\nu\}_{\nu > 0}$, one may expect that the transported field will also develop small scales. This logic holds true for a passive scalar, where the scalar obediently follows the Lagrangian paths of the velocity field $\bs{U}$. However, in the passive vector case, the introduction of pressure complicates matters and may resist the formation of small-scale structures if the velocity field $\bs{U}$ is not carefully designed. This makes addressing Question \ref{Q: anom dissp div free} more challenging compared to its counterpart concerning anomalous dissipation in a passive scalar. Therefore, before addressing the difficult Question \ref{Q: anom dissp div free} about anomalous dissipation in a passive vector field, one can ask a similar question about enhanced dissipation.

\subsection{Enhanced dissipation in a vector field}
Enhanced dissipation refers to a phenomenon where mixing, transport of scalar and momentum, or energy dissipation in a fluid occurs at a much faster rate compared to pure molecular diffusion \citep{zelati2020relation}. This phenomenon has been widely studied recently in the mathematical fluid dynamics community. In a pure diffusion case, the rate of energy dissipation typically scales linearly with viscosity $\nu$ for small viscosity values. However, in the enhanced dissipation scenario, the rate of energy dissipation increases and potentially follow a mixed of power-law and logarithmic scaling $\nu^\alpha (\log \nu^{-1})^\beta$, where 
\begin{align}
\alpha \in [0, 1) \text{ and } \beta \in \mathbb{R}  \text{ if } \alpha > 0 \text{ and } \beta \leq 0 \text{ if } \alpha = 0.
\label{alpha beta choice}
\end{align}
With $\alpha = \beta = 0$, we arrive at anomalous dissipation. In the context of passive vector equations (\ref{eqn: conv diff u}), one can ask if the underlying advecting velocity field $\bs{U}$ is dissipation-enhancing, does that also mean the same is true for the advected velocity field $\bs{u}$? More precisely, we ask the following question.

%In the context of passive vector equations (\ref{eqn: conv diff u}), one can ask if the underlying advecting velocity field $\bs{U}$ is dissipation enhancing, does that also means the same is true for the advected velocity field $\bs{u}$. More precisely, it leads to the following question.

 %Therefore, it is not clear how one would adapts the strategies used in the case of passive scalar case to passive vector case. In fact, we believe that the flow designs used in above mentioned studies, such as alternating shear flows or checkered board flow, will simply not work in the case of a passive vector equation. As such answering the question raised above is hard and we do not know how to answer it.

%This simple reasoning would be true in the case of a passive scalar, where the scalar is slaved to the velocity field.

%Therefore, if it is not a well thought out flow then you are in trouble.

%Answering this question is still quite challenging compared to questions regarding anomalous dissipation in passive scalar. The simple reason in because of the presence of pressure. If you form small scales in $U$, while a passive scalar will just obey the same but a divergence-free vector field may not. As a result it is not clear how the strategies from papers XYZ will apply here.

\begin{question}[Enhanced dissipation in a divergence-free vector field]
\label{enhanced dissp}
Given the admissible choices of constants $\alpha$ and $\beta$ as in \eqref{alpha beta choice}, can one design smooth divergence-free vector fields $\{\bs{U}_\nu\}_{\nu_0 > \nu > 0}$ obeying
\begin{align}
\nu \langle |\nabla \bs{U}_{\nu}|^2 \rangle \sim \nu^{\alpha} (\log \nu^{-1})^\beta,
\end{align}
such that the corresponding family of solutions $\{\bs{u}_\nu\}_{\nu_0 > \nu > 0}$ to (\ref{eqn: conv diff u}) satisfies
\begin{align}
\nu \langle |\nabla \bs{u}_{\nu}|^2 \rangle \geq c \nu^{\alpha} (\log \nu^{-1})^\beta \quad \text{for all} \quad \nu \in (0, \nu_0),
\end{align}
where $c > 0$ is some constant?
\end{question}
In this paper, we answer this question affirmative with  $\alpha = 0$ and $\beta = -2$ in a Couette flow system. Therefore, both the advecting velocity fields $\{\bs{U}_\nu\}$ and passive velocity fields $\{\bs{u}_\nu\}$ are anomalous dissipating and only miss by a factor of logarithm.

\subsection{Couette Flow}
\label{sec: flow config}
In this paper, we focus on the well-known Couette flow configuration, which serves as a paradigmatic system for studying turbulence within a shear layer. It is arguably one of the most extensively studied flow configurations due to its simple geometry, which facilitates both experimental design and direct numerical simulations. The Couette flow is the flow of fluid between two parallel plates, where the bottom plate is stationary and to top plate moves with a prescribed velocity. We solve the equations (\ref{eqn: conv diff u}{\color{blue}a-b}) in the domain 
\begin{align}
\Omega \coloneqq \mathbb{T}_{L_1} \times [-1/2, 1/2] \times \mathbb{T}_{L_3}.
\end{align}
 \begin{figure}
\centering
 \includegraphics[scale = 0.5]{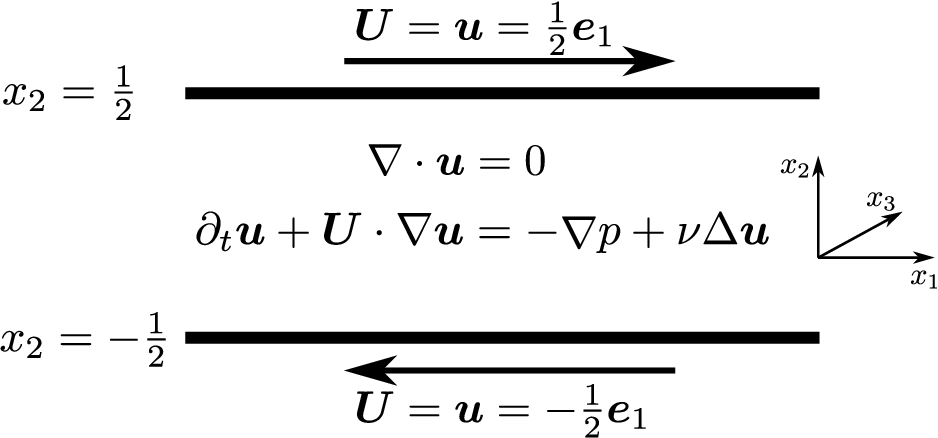}
 \caption{shows the schematic of the flow configuration.}
 \label{fig: Couette}
\end{figure}
Figure \ref{fig: Couette} shows a schematic of the flow configuration. We employ a Cartesian coordinate system $(x_1, x_2, x_3)$, where $x_1$ and $x_3$ represent the periodic direction directions with length $L_1$ and $L_3$ respectively, while $x_2$ denotes the wall-normal direction. The velocity field $\bs{u}$ and $\bs{U}$ satisfy the same initial conditions:
%\begin{align}
%\bs{U}(0, x_1, x_2, x_3) = \bs{u}(0, x_1, x_2, x_3) = \bs{u}_0(x_1, x_2, x_3), \quad \forall (x_1, x_2, x_3) \in \Omega
%\label{flow config: IC}
%\end{align}
\begin{align}
\bs{U}(t, \cdot) \to \bs{u}_0, \quad \bs{u}(t, \cdot) \to \bs{u}_0 \qquad \text{in } L^2(\Omega) \text{ as } t \to 0^+,  
\label{flow config: IC}
\end{align}
and boundary conditions:
\begin{align}
\bs{U}(t, x_1, -1/2, x_3) = \bs{u}(t, x_1, -1/2, x_3) = -\bs{e}_1/2 \qquad \bs{U}(t, x_1, 1/2, x_3) = \bs{u}(t, x_1, 1/2, x_3) = \bs{e}_1/2,
\label{flow config: BC}
\end{align}
for all $t > 0$ and $(x_1, x_3) \in \mathbb{T}_{L_1} \times \mathbb{T}_{L_3}$, where $\bs{e}_1$ is the unit vector in the $x_1$ direction.

In rest of the paper, we will consider that the velocity field $\bs{U}$ belongs to $ L^2([0, \infty); H^1(\Omega)) \cap L^\infty([0, \infty) \times \Omega)$.  Furthermore, we assume that $\bs{u}_0 \in H^1(\Omega)$. Therefore, there is a unique  solution $\bs{u} \in L^\infty([0, \infty); H^1(\Omega)) \cap L^2([0, \infty); H^2(\Omega))$ of the equations (\ref{eqn: conv diff u}{\color{blue}a-b}) with initial condition (\ref{flow config: IC}) and boundary conditions (\ref{flow config: BC}).

Next, we define 
\begin{align}
\varepsilon_{\bs{U}} \coloneq \nu \langle |\nabla \bs{U}|^2 \rangle, \qquad \qquad \varepsilon_{\bs{u}} \coloneq \nu \langle |\nabla \bs{u}|^2 \rangle
\label{flow config: dissp U u}
\end{align}
to be the energy dissipation in the velocity $\bs{U}$ and $\bs{u}$ respectively. Recall, the angle brackets denote the long-time volume average as defined in (\ref{prb: long time avg}). For convenience, we also define the long-time horizontal average, which we denote using overbar, as
\begin{align}
\ol{[\, \cdot \,]} \coloneqq \limsup_{T \to \infty} \frac{1}{T} \int_{0}^{T}  \dashint_{\mathbb{T}_{L_1} \times \mathbb{T}_{L_3}} [\, \cdot \,] \, {\rm d} \bs{s} \, {\rm d} t \qquad \text{and} \qquad \dashint_{\mathbb{T}_{L_1} \times \mathbb{T}_{L_3}} [\, \cdot \,] \, {\rm d} \bs{s} \coloneqq \frac{1}{L_1 \, L_3} \int_{\Omega} [\, \cdot \,] \, {\rm d} x_1 {\rm d} x_3.
\label{prb: long time avg}
\end{align}
Next, we state the main results of our paper, all of which are in the Couette flow setting. The first result provides an upper bound on the dissipation in $\bs{u}$ in terms of the dissipation in $\bs{U}$.
\begin{theorem}[Main result I: Upper bound]
Let the energy dissipation in $\bs{U}$ and $\bs{u}$ be given by (\ref{flow config: dissp U u}). Then we have the following upper bound
\begin{align}
\varepsilon_{\bs{u}} \leq 4 (\varepsilon_{\bs{U}})^{\frac{1}{3}}.
\end{align}
Therefore, if $\{\bs{U}_{\nu}\}_{\nu > 0}$, a family of divergence-free velocity fields, obeys $\nu \langle |\nabla \bs{U}_\nu|^2 \rangle \sim \nu^{\alpha} (\log \nu^{-1})^\beta$  then the corresponding solutions $\{\bs{u}_{\nu}\}_{\nu > 0}$ to equations (\ref{eqn: conv diff u}{\color{blue}a-b}) will obey $\nu \langle |\nabla \bs{u}_\nu|^2 \rangle \leq c \nu^{\frac{\alpha}{3}} (\log \nu^{-1})^{\frac{\beta}{3}}$ for some positive constant $c$.
\label{thm: upper bound}
\end{theorem}

\noindent
Our second result is a construction, we prove the following:
\begin{theorem}[Main result II: Construction]
There exists a family of smooth divergence-free velocity field $\{\bs{U}_{\nu}\}_{0 < \nu < 1/500}$ with
$\nu \langle |\nabla \bs{U}_\nu|^2 \rangle \sim (\log(\nu^{-1}))^2$ such that the solutions to equations (\ref{eqn: conv diff u}{\color{blue}a-b}) satisfy $\nu \langle |\nabla \bs{u}_\nu|^2 \rangle \geq c (\log(\nu^{-1}))^2$.
\label{thm: main thm}
\end{theorem}
\begin{remark}
The proof of Theorem \ref{thm: upper bound} relies on the method by \cite{seis2015bound}. The result for values other than $\alpha = 0$ and $\beta = 0$ may not be sharp as it does not seem physical for the advected velocity field $\bs{u}$ to be dissipating more energy than the advecting velocity field $\bs{U}$. This raises the following question:
\end{remark}
\begin{question}[Improved upper bound]
If $\{\bs{U}_{\nu}\}_{\nu > 0}$ obeys $\nu \langle |\nabla \bs{U}_\nu|^2 \rangle \sim \nu^\alpha (\log \nu)^\beta$ with $\alpha$ and $\beta$ as given in (\ref{alpha beta choice}) then is it true that the corresponding solutions $\{\bs{u}_{\nu}\}_{\nu > 0}$ to equations (\ref{eqn: conv diff u}{\color{blue}a-b}) will obey $\nu \langle |\nabla \bs{u}_\nu|^2 \rangle \leq c \nu^{\alpha} (\log \nu)^\beta$ for some positive constant $c$?
\end{question}
The proof of Theorem \ref{thm: main thm} is based on a variational principle for the energy dissipation rate in the velocity field $\bs{u}$ derived in Section \ref{sec: var prin}, which is similar to variational approaches  recently used in the context of advection-diffusion in a scalar for problems related to optimal heat transport \citep{doering2019optimal, kumar2022three}. The main idea is to use a branching flow construction to a give a proof of Theorem \ref{thm: main thm}. However, we also provide a subptimal answer to Question \ref{enhanced dissp} using a construction based on convection rolls with $\alpha = 1/3$ and $\beta = 0$. This construction serves a dual purpose: it demonstrates the application of the variational principle (before going to the more complex branching flow construction) and is also physically motivated, as convection rolls are commonly observed in fluid dynamics and are effective in heat and momentum transfer. Consequently, this work may inspire further research into finding exact solutions of the Navier–Stokes equations based on convection rolls in a Couette flow configuration, both from mathematical analysis and computational side, as the convection rolls seem more tractable than the branching flows.

%However, the derivation of such a variational principle dates back at least to \citet{ortiz1985variational} for the development of an upwind finite element scheme.

\subsection{Comparison with experiments and numerical observation}
The single most important aspect of our result in Theorem \ref{thm: main thm} is that the energy dissipation scaling $(\log \nu^{-1})^{-2}$ is the predictions from the prominent phenomenological theories of turbulent channel flow (or wall-bounded turbulence in general) and has been validated in experiments to a considerable degree. From the vast literature on this subject, we note only the most important references \citep{schlichting2016boundary, smits2011high, yakhot2010scaling, marusic2010wall}. Beyond the statistical property, our flow design in Theorem \ref{thm: main thm} shares several structural similarities. We illuminate on these two features in the paragraphs below.

The near anomalous dissipation scaling of the energy dissipation in a channel flow is a prediction of the so-called ``law of the wall'' or ``logarithmic law.'' This law asserts that the mean velocity of the flow near the wall scales logarithmically with the distance from the wall, a characteristic behavior observed in turbulent flows. It describes three distinct layers in turbulent flow: (i) the viscous layer (closest to the wall), (ii) the outer layer (further away from the wall), and (iii) the overlap layer. In the viscous layer, viscous effects dominate, and the velocity profile in this region is independent of the channel height. In the outer layer, the velocity is influenced only by the channel height and does not account for viscosity. Finally, in the overlap region, the velocity is independent of both viscosity and channel height, leading to a logarithmic velocity profile. A consequence of this logarithmic profile is that the energy dissipation scales as
\begin{align}
\nu \langle |\nabla \bs{u}|^2 \rangle \sim \frac{1}{4} \frac{\kappa^2}{(\log \nu^{-1})^2} \approx \frac{0.042}{(\log \nu^{-1})^2},
\label{drag coefficient Couette flow}
\end{align}
where $\kappa \approx 0.41$ is know as the ``von K\'arm\'an'' constant, a universal constant in wall-bounded turbulent flows \citep{schlichting2016boundary}. We caution the reader that the result in \eqref{drag coefficient Couette flow} is often stated in the engineering and physics literature in terms of the friction factor (which is related to energy dissipation) and the Reynolds number $Re = \nu^{-1} U_c H$, where $U_c$ is the velocity of the walls ($1/2$ in our case) and $H$ is the half channel width (also $1/2$ in our case).

%The second aspect we want to highlight in our flow design in Theorem \ref{thm: main thm} is the self-similar nature of the flow which is made up of overlapping eddies and is crucial is getting the logarithmic decay of  the energy dissipation. Interestingly enough, the self-similar hierarchical flow structures is also common in wall turbulence. In an important work Townsend hypothesized that the flow across the log-layer (see discussion above) as a self-similar population of eddies of different sizes attached to the wall. The similar structures have been observed in DNS, numerically constructed invariant solutions, resolvent analysis, proper orthogonal decomposition and the hypothesis is well corborated by experiments. All in all, our construction is therefore well related to physicaly observed flow and is not just a mathematical construct.

The second aspect we wish to emphasize regarding our flow design in Theorem \ref{thm: main thm} is the self-similar nature of the construction, which is characterized by overlapping eddies and is crucial for achieving the logarithmic decay of energy dissipation. We note that the self-similar hierarchical structure is also prevalent in wall turbulence. In a seminal work, Townsend hypothesized that the flow within the log layer comprises a self-similar population of eddies of varying sizes that are attached to the wall \citep{townsend1976structure}. Such structures have been observed in direct numerical simulations (DNS) \citep{lozano2012three, hwang2015statistical}, numerically constructed invariant solutions \citep{yang2019exact}, resolvent analysis \citep{mckeon2019self}, and proper orthogonal decomposition of data from pipe flow \citep{hellstrom2016self}. Given the scaling results of our velocity field construction and the similarity in flow structures, it is evident that our construction is closely related to physically observed flows and is not merely a mathematical construct.

\subsection{Organization of the paper}
In Section \ref{sec: upper bound}, we provide a proof of Theorem \ref{thm: upper bound} by establishing an upper bound on the energy dissipation. Section \ref{sec: var prin} is dedicated to deriving a variational principle for the rate of energy dissipation. We present a construction based on convection rolls in Section \ref{sec: convection rolls}. Next, we prove Theorem \ref{thm: main thm} in Section \ref{sec: branching flows}, which relies on branching flows. Finally, we conclude with a discussion and future outlook in Section \ref{sec: conclusion}.

\begin{center}
\subsection*{Acknowledgement}
\end{center}
A.K. thanks Sergei Chernyshenko for insightful discussions on the law of the wall. A.K. also thanks Theodore Drivas and Vlad Vicol for valuable suggestions.

\vspace{0.5cm}
\section{Upper Bound}
\label{sec: upper bound}
As the velocity field $\bs{U} \in L^2([0, \infty); H^1(\Omega)) \cap L^\infty([0, \infty) \times \Omega)$. Therefore, there is a unique strong solution $\bs{u}$ of the equations (\ref{eqn: conv diff u}{\color{blue}a-b}). Taking the dot product of equation (\ref{eqn: 1 prb momentum}) with $\bs{u}$ then gives
\begin{align}
\frac{1}{2} \partial_t |\bs{u}|^2 + \frac{1}{2} \div(\bs{U} |\bs{u}|^2) = - \div(\bs{u} p) + \nu \sum_{i, j} \partial_j (u_i \partial_j u_i) - \nu |\nabla \bs{u}|^2.
\end{align}
Taking a volume average of above equation leads to
\begin{align}
\frac{1}{2} \frac{d}{dt} \dashint_{\Omega} |\bs{u}|^2 \, {\rm d} \bs{x} = \frac{\nu}{2} \dashint_{\mathbb{T}_{L_x} \times \mathbb{T}_{L_y}} \left[\left. \partial_2 u_1 \right|_{x_2=1/2} + \left. \partial_2 u_1 \right|_{x_2=-1/2} \right] {\rm d} \bs{s} - \nu \dashint_{\Omega} |\nabla \bs{u}|^2 \, {\rm d} \bs{x}.
\end{align}
Finally, noting that $\bs{u} \in L^\infty([0, \infty); L^2(\Omega))$, performing a long-time average leads to
\begin{align}
\nu \langle |\nabla \bs{u}|^2 \rangle = \frac{\nu}{2} \ol{\left. \partial_2 u_1 \right|_{x_2=1/2} } + \frac{\nu}{2} \ol{\left. \partial_2 u_1 \right|_{x_2=-1/2}}.
\label{eqn: dissp for 1}
\end{align}
By performing the long-time horizontal average of the $x_1$ component of equation (\ref{eqn: 1 prb momentum}) at any level $x_2$ gives 
\begin{align}
\nu \ol{\partial_2 u_1} - \ol{U_2 u_1} = const \quad \text{for all } x_2 \in [-1/2, 1/2].
\label{eqn: dissp for 2}
\end{align}
Combining (\ref{eqn: dissp for 1}) and (\ref{eqn: dissp for 2}), we get
\begin{align}
\nu \langle |\nabla \bs{u}|^2 \rangle = \nu \ol{\partial_2 u_1} - \ol{U_2 u_1}  \quad \text{for all } x_2 \in [-1/2, 1/2].
\label{eqn: dissp for 3}
\end{align}
Therefore, the equality also holds when the right-hand side in (\ref{eqn: dissp for 3}) is replaced by its average over the layer $$\Gamma_{\delta} \coloneqq [-1/2, -1/2 + \delta]$$ of thickness $\delta$ near the bottom wall leading to
\begin{align}
\nu \langle |\nabla \bs{u}|^2 \rangle & = \nu \ol{\dashint_{ \Gamma_\delta}\partial_2 u_1}  - \ol{\dashint_{\Gamma_\delta} U_2 u_1 }  \qquad \text{where} \quad \dashint_{\Gamma_{\delta}} [\, \cdot \,] \coloneqq \frac{1}{\delta}\int_{x_2 \in \Gamma_\delta} [\, \cdot \,] \, {\rm} d x_2. 
\label{eqn: dissp for 4}
\end{align}
%\nonumber \\
%& = \frac{\nu}{\delta} \, \ol{u_1 |_{x_2 = \delta}}  - \ol{\dashint_{\Gamma_\delta} U_2 u_1 }
Next, we have the following estimate:
\begin{align}
\left| \, \ol{\dashint_{ \Gamma_\delta}\partial_2 u_1} \, \right| \leq \left(\, \ol{ \dashint_{\Gamma_\delta} (\partial_2 u_1)^2 } \, \right)^{\frac{1}{2}} \leq  \nu^{-1/2} \delta^{-1/2} (\varepsilon_{\bs{u}})^{\frac{1}{2}}.
\label{eqn: dissp est 1}
\end{align}
We also have
\begin{align}
\left| \ol{\dashint_{\Gamma_\delta} U_2 u_1} \right| 
& \leq \left( \, \ol{\dashint_{\Gamma_\delta} U_2^2 } \, \right)^{\frac{1}{2}} \left( \, \ol{ \dashint_{\Gamma_\delta} u_1^2 } \, \right)^{\frac{1}{2}} \nonumber \\
& \leq \delta^2 \left( \, \ol{ \dashint_{\Gamma_\delta} (\partial_2 U_2)^2 \, }\right)^{\frac{1}{2}} \left(\, \ol{ \dashint_{\Gamma_\delta} (\partial_2 u_1)^2 } \, \right)^{\frac{1}{2}} \nonumber \\
& \leq  \delta \nu^{-1} (\varepsilon_{\bs{U}})^{\frac{1}{2}} (\varepsilon_{\bs{u}})^{\frac{1}{2}}.
\label{eqn: dissp est 2}
\end{align}
%\begin{align}
%\dashint_{\Gamma_\delta} U_2 u_1
%& \leq \left(\dashint_{\Gamma_\delta} U_2^2 \right)^{\frac{1}{2}} \left(\dashint_{\Gamma_\delta} u_1^2 \right)^{\frac{1}{2}} \nonumber \\
%& \leq \delta^2 \left(\dashint_{\Gamma_\delta} (\partial_2 U_2)^2 \right)^{\frac{1}{2}} \left(\dashint_{\Gamma_\delta} (\partial_2 u_1)^2 \right)^{\frac{1}{2}}
%\end{align}
Combining (\ref{eqn: dissp for 4}) with (\ref{eqn: dissp est 1}) and (\ref{eqn: dissp est 2}) leads to
\begin{align}
& \varepsilon_{\bs{u}} \leq \nu^{1/2} \delta^{-1/2}  (\varepsilon_{\bs{u}})^{\frac{1}{2}} + \delta \nu^{-1} (\varepsilon_{\bs{U}})^{\frac{1}{2}} (\varepsilon_{\bs{u}})^{\frac{1}{2}} \nonumber \\
\implies & \varepsilon_{\bs{u}} \leq \left(\nu^{1/2} \delta^{-1/2}  + \delta \nu^{-1} (\varepsilon_{\bs{U}})^{\frac{1}{2}} \right)^2.
\end{align}
Finally, selecting $$\delta = \frac{\nu}{(\varepsilon_{\bs{U}})^{\frac{1}{3}}}$$
gives
\begin{align}
\varepsilon_{\bs{u}} \leq 4 (\varepsilon_{\bs{U}})^{\frac{1}{3}}. 
\end{align}
\section{Variational principle for the energy dissipation rate}
\label{sec: var prin}
To prove Theorem \ref{thm: main thm}, we will design a family of time-independent divergence-free velocity field $\{\bs{U}_{\nu}\}_{\nu > 0}$. As a result, in long-time averages of the solutions of (\ref{eqn: conv diff u}), any dependence on the initial data $\bs{u}_0$ is ultimately lost, and the long-time averages depend solely on the solutions of the corresponding steady equations:
\begin{subequations}
\begin{align}
& \nabla \cdot \bs{u} = 0, 
\label{eqn: steady div free}\\
 & \bs{U} \cdot \nabla \bs{u} = - \nabla p + \nu \Delta \bs{u}.
\label{eqn: steady momentum}
\end{align}
\label{eqn: steady conv diff u}
\end{subequations}
When the choice of $\bs{U}$ involves a combination of translation and rigid body motion, the resulting equations are known as Oseen's equations \citep{galdi2011introduction}. However, in our study, we do not limit $\bs{U}$ to just these two forms. The construction that follows in the next two sections, the design of the velocity fields $\{\bs{U}_{\nu}\}_{\nu > 0}$ is two dimensional. 
Because of this simplification, from here onward, without loss of generality, our domain will be
\begin{align}
\Omega = \mathbb{T}_{L_1} \times [-1/2, 1/2],
\end{align}
and the angle brackets will only mean the volume averages
\begin{align}
\langle [\, \cdot \,] \rangle = \dashint_{\Omega} [\, \cdot \,] \, {\rm d} \bs{x}.
\end{align}
The velocity fields $\bs{U}$ and $\bs{u}$ satisfy the same boundary conditions on the top and bottom walls
\begin{align}
\bs{U}(x_1, -1/2) = \bs{u}(x_1, -1/2) = -\bs{e}_1/2, \qquad \bs{U}(x_1, 1/2) = \bs{u}(x_1, 1/2) = \bs{e}_1/2,
\label{flow config: steady BC}
\end{align}
for all $x_1 \in \mathbb{T}_{L_1}$.
\begin{remark}
From the broader perspective of the time-dependent advection-diffusion equation, the velocity field $\bs{U}$ may not initially satisfy the given conditions. To resolve this  discrepancy, the velocity field $\bs{U}$ is designed to smoothly transition from the initial condition over the first unit of time. After this initial period, $\bs{U}$ assumes the desired autonomous form.
\end{remark}
We assume in rest of the section we assume that $\bs{U} \in L^\infty(\Omega)$. Therefore, there is a unique weak solution $\bs{u} \in H^1(\Omega)$ to (\ref{eqn: steady conv diff u}) satisfying the boundary conditions (\ref{flow config: steady BC}). Also, we denote the Leary projector by $\mathbb{P}$ and the $\Delta^{-1}$ to be the inverse Laplace operator with respect to homogeneous boundary conditions on the walls $x_2 = \pm \frac{1}{2}$.
\begin{proposition}[Variational principle for energy dissipation rate]
Let $\bs{U} \in L^\infty(\Omega)$ be a weakly divergence-free vector field. Then the energy dissipation in the solution $\bs{u}$ of (\ref{eqn: steady conv diff u}) satisfies
\begin{align}
\nu \langle |\nabla \bs{u}|^2 \rangle - \nu = \max_{\substack{\wt{\bs{v}} \in H^1_0(\Omega) \\ \div \wt{\bs{v}} = 0}} \left(- 2 \langle U_2 \wt{v}_1 \rangle  - \nu \langle |\nabla \wt{\bs{v}}|^2 \rangle - \frac{1}{\nu} \langle |\nabla \Delta^{-1} \mathbb{P} \bs{U} \cdot \nabla \wt{\bs{v}}|^2 \rangle\right).
\label{var prin: variational principle}
\end{align}
\end{proposition}
\begin{proof}
We decompose the velocity field $\bs{u}$ into a linear part and a homogeneous part:
\begin{align}
\bs{u} = x_2 \bs{e}_1 + \bs{v}.
\label{var prin: decomposition of u}
\end{align}
Using (\ref{var prin: decomposition of u}) in (\ref{eqn: steady conv diff u}) and (\ref{flow config: steady BC}), it is clear that $\bs{v}$ satisfies
\begin{subequations}
\begin{align}
& \nabla \cdot \bs{v} = 0, 
\label{var prin: eqn for v div-free}
\\
& U_2 \bs{e}_1 + \bs{U} \cdot \nabla \bs{v} = - \nabla p +  \nu \Delta \bs{v}.
\label{var prin: eqn for v mom}
\end{align}
\label{var prin: eqn for v}
\end{subequations}
along with the homogeneous version of the boundary conditions.
From the decomposition (\ref{var prin: decomposition of u}), we observe that 
\begin{align}
\nu \langle |\nabla \bs{u}|^2 \rangle = \nu \langle |\nabla \bs{v}|^2 \rangle + \nu.
\label{var prin: dissp formula 1}
\end{align}
Also, taking the dot product of (\ref{var prin: eqn for v mom}) with $\bs{v}$ and performing a volume average leads to 
\begin{align}
\nu \langle |\nabla \bs{v}|^2 \rangle =  - \langle U_2 v_1 \rangle
\end{align}
Therefore, we can also write
\begin{align}
\nu \langle |\nabla \bs{u}|^2 \rangle = - \langle U_2 v_1 \rangle + \nu.
\label{var prin: dissp formula 2}
\end{align}
Next, consider two sets of PDEs,
\begin{subequations}
\begin{align}
& \nabla \cdot \wt{\bs{v}} = 0, \\
& U_2 \bs{e}_1 + \bs{U} \cdot \nabla \ol{\bs{v}} = - \nabla \wt{p} + \nu \Delta \wt{\bs{v}},
\label{var prin: v overbar mom}
\end{align}
\label{var prin: v overbar}
\end{subequations}
and
\begin{subequations}
\begin{align}
& \nabla \cdot \ol{\bs{v}} = 0, \\
& \bs{U} \cdot \nabla \wt{\bs{v}} = - \nabla \ol{p} + \nu \Delta \ol{\bs{v}}, 
\label{var prin: v tilde mom}
\end{align}
\label{var prin: v tilde}
\end{subequations}
where $\ol{\bs{v}}, \wt{\bs{v}}: \Omega \to \mathbb{R}^2$ satisfies the zero velocity boundary conditions, i.e.,
\begin{align}
\ol{\bs{v}}, \; \wt{\bs{v}} = \bs{0}, \qquad x_2 = \pm \frac{1}{2}.
\end{align}
This is a standard symmetrization procedure previously used in the studies of optimal heat transport \citep{doering2019optimal}. Now, from the standard existence uniqueness theory, there are unique solutions $\wt{\bs{v}}, \ol{\bs{v}} \in H^1_0(\Omega)$ to PDEs (\ref{var prin: v overbar}) and (\ref{var prin: v tilde}). Moreover, it is clear that we can write our solution $\bs{v}$ to equations (\ref{var prin: eqn for v}) as
\begin{align}
\bs{v} = \ol{\bs{v}} + \wt{\bs{v}}.
\label{var prin: v vtilde vbar}
\end{align}
Taking the dot product of equation (\ref{var prin: v tilde mom}) with $\wt{\bs{v}}$ leads to 
\begin{align}
\langle \nabla \ol{\bs{v}} \colon \nabla \wt{\bs{v}} \rangle = 0,
\label{var prin: ortho v over tilde}
\end{align}
where in the index notation, we have $\nabla \ol{\bs{v}} \colon \nabla \wt{\bs{v}} = \partial_j \ol{v}_i \partial_j \wt{v}_i$.
Next, we take the dot product of the equation (\ref{var prin: v overbar mom}) with $\ol{\bs{v}}$ and integrate which leads to 
\begin{align}
\langle U_2 \ol{v}_1 \rangle = 0.
\label{var prin: transport overbar zero}
\end{align} 
Substituting (\ref{var prin: v vtilde vbar}) in (\ref{var prin: dissp formula 1}) and then using (\ref{var prin: ortho v over tilde}) leads to 
\begin{align}
\nu \langle |\nabla \bs{u}|^2 \rangle & = \nu \langle |\nabla \wt{\bs{v}}|^2 \rangle  + \nu \langle |\nabla \ol{\bs{v}}|^2 \rangle + \nu \nonumber \\
& = \nu \langle |\nabla \wt{\bs{v}}|^2 \rangle + \frac{1}{\nu} \langle |\nabla \Delta^{-1} \mathbb{P} \bs{U} \cdot \nabla \wt{\bs{v}}|^2 \rangle  + \nu.
\label{var prin: dissp formula 3}
\end{align}
 Using (\ref{var prin: transport overbar zero}) in (\ref{var prin: dissp formula 2}) gives
\begin{align}
\nu \langle |\nabla \bs{u}|^2 \rangle = - \langle U_2 \wt{v}_1 \rangle + \nu.
\label{var prin: dissp formula 4}
\end{align}
Finally, multiplying (\ref{var prin: dissp formula 4}) with two and then subtracting (\ref{var prin: dissp formula 3}) yields
\begin{align}
\nu \langle |\nabla \bs{u}|^2 \rangle - \nu = F(\wt{\bs{v}}),
\end{align}
where
\begin{align}
F(\wt{\bs{v}}) \coloneqq - 2 \langle U_2 \wt{v}_1 \rangle  - \nu \langle |\nabla \wt{\bs{v}}|^2 \rangle - \frac{1}{\nu} \langle |\nabla \Delta^{-1} \mathbb{P} \bs{U} \cdot \nabla \wt{\bs{v}}|^2 \rangle.  
\end{align}
Next, we consider the following variational problem:
\begin{align}
\sup_{\substack{\wt{\bs{v}} \in H^1_0(\Omega) \\ \div \wt{\bs{v}} = 0}} F(\wt{\bs{v}}).
\end{align}
We see that $F$ is concave so it has a unique maximizer. Moreover, the maximizer $\wt{\bs{v}}$ solves the Euler--Lagrange equation
\begin{align}
U_2 \bs{e}_1 + \frac{1}{\nu} \bs{U} \cdot \nabla \Delta^{-1} (\mathbb{P} \bs{U} \cdot \nabla \wt{\bs{v}}) = - \nabla q +  \nu \Delta \wt{\bs{v}}
\end{align}
which is basically the PDE (\ref{var prin: v overbar mom}). Therefore, the variational principle (\ref{var prin: variational principle}) follows. Finally, note that one could have also used a method based on Lagrange multiplier to derive this variational principle. See, for example, \citep{song2023bounds} in the context of heat transfer.
\end{proof}
\section{Convection Rolls Based Design}
\label{sec: convection rolls}
In this section, we demonstrate the existence of a family of divergence-free velocity fields $\{\bs{U}_\nu\}_{\nu > 0}$ with energy dissipation scaling as $\nu \langle |\nabla \bs{U}_{\nu}|^2 \rangle \sim \nu^{1/3}$, such that for the corresponding passive vector fields, the energy dissipation scales at least as $\nu \langle |\nabla \bs{u}_{\nu}|^2 \rangle \gtrsim \nu^{1/3}$. For brevity, we will omit $\nu$ from the subscript in the remainder of this section. The first step involves constructing a candidate for the convecting velocity field $\bs{U}$.
Once $\bs{U}$ is constructed, we turn to the variational principle (\ref{var prin: variational principle}) derived in the previous section to obtain a bound on the energy dissipation in $\bs{u}$. Therefore, the second step consists of making a `good choice' of the divergence-free test vector field $\wt{\bs{v}} \in H^1_0(\Omega)$, which then provides the following bound
\begin{align}
\nu \langle |\nabla \bs{u}|^2 \rangle - \nu \geq I + II + III,
\label{proof main thm: unknown lower bound}
\end{align}
where
\begin{align}
I \coloneqq - 2 \langle U_2 \wt{v}_1 \rangle, \qquad II \coloneqq  - \nu \langle |\nabla \wt{\bs{v}}|^2 \rangle, \qquad III \coloneqq - \frac{1}{\nu} \langle |\nabla \Delta^{-1} \mathbb{P} \bs{U} \cdot \nabla \wt{\bs{v}}|^2 \rangle.
\label{proof main thm: I II III}
\end{align}
Finally, the last task is to estimate all the terms $I$, $II$ and $III$.
%The proof of the main theorem consists of two steps. In the first step, we make a choice of the vector field $\bs{U}$ which belongs to class $\mathcal{E}_{\frac{1}{3}}$. In the second step we make a choice of $\wt{\bs{v}}$ in the variational principle (\ref{var prin: variational principle}) to derive a lower bound on the rate of energy dissipation $\varepsilon_{\bs{u}} = \nu \langle |\nabla \bs{u}|^2 \rangle$ and show $\bs{u} \in \mathcal{E}_{\frac{1}{3}}$. Next, we carry out these two steps.
\subsection{Choice of the velocity field $\bs{U}$}
Our choice of the velocity field  $\bs{U}$ is the sum of two divergence-free velocity fields:
\begin{align}
\bs{U} = \bs{U}^m + \bs{U}^c.
\end{align} 
The velocity field $\bs{U}^m$ is a unidirectional flow in the $x_1$ direction that satisfies the same boundary conditions as $\bs{U}$. This velocity field can be thought of as a mean flow that is flat in the middle of the domain and has sharp gradients near the top and bottom boundaries. The velocity field $\bs{U}^c$  satisfies the homogeneous version of the boundary conditions and consists of convection rolls inclined at an angle of $3\pi/4$ from the $x_1$ direction. Constructions based on vertical convection rolls have previously been used in studies of optimal heat transport \citep{iyer2021bounds, souza2020wall}. The three free parameters in our choice of $\bs{U}$ are $\delta$, the boundary layer thickness, $k$, the horizontal wavenumber of the rolls, and an amplitude $\mathcal{A}$. These parameters are chosen based on the viscosity $\nu$.
 \begin{figure}
\centering
 \includegraphics[scale = 0.8]{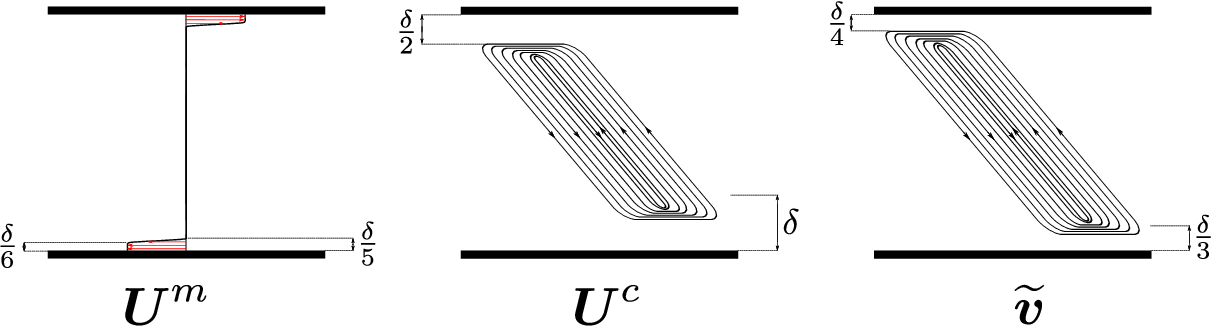}
 \caption{illustrates the velocity fields $\bs{U}^m$, $\bs{U}^c$ and $\wt{\bs{v}}$ defined in  (\ref{def Um}), (\ref{def Uc}) and (\ref{def v tilde}), respectively. The two important things to notice are as follows: (1) the support of $\nabla \bs{U}^m$, the region where rolls in $\bs{U}^c$ and $\wt{\bs{v}}$ folds are disjoint, (2) the rolls in $\wt{\bs{v}}$ are taller than the those in $\bs{U}^c$.}
 \label{fig: velocity}
\end{figure}

To construct the velocity field $\bs{U}$, we start by defining a cutoff function. Given two numbers  $\delta_1$ and $\delta_2$ such that $0 < \delta_1 < \delta_2 < 1/4$, we consider a cutoff function $\chi(\cdot; \delta_1, \delta_2): [-1/2, 1/2] \to [0, 1]$ such that
\begin{align}
\chi(x; \delta_1, \delta_2) = 
\begin{cases}
1 \qquad \text{if} \quad x \in \left[-\frac{1}{2} + \delta_2, \frac{1}{2} - \delta_2 \right], \\[5pt]
0 \qquad \text{if} \quad x \in \left[-\frac{1}{2}, -\frac{1}{2} + \delta_1 \right] \cup \left[\frac{1}{2} - \delta_1, \frac{1}{2} \right],
\end{cases}
\end{align}
and it satisfies the following estimates:
\begin{align}
\norm{\nabla^i \chi(\cdot; \delta_1, \delta_2)}_{L^\infty} \leq C_i \max\left\{\frac{1}{\delta_1^i}, \, \frac{1}{(\delta_2 - \delta_1)^i}\right\} \qquad \forall i \in \mathbb{N}, 
\end{align}
where $C_i$ are constants independent of $\delta_1$ and $\delta_2$.

Next, for some $\delta \in (0, 1/4)$, we define our mean velocity field $\bs{U}^{m}: \Omega \to \mathbb{R}^2$ as
\begin{align}
\bs{U}^m(\bs{x}) \coloneqq \ol{U}^m(x_2) \bs{e}_1,
\label{def Um}
\end{align}
where $\bs{e}_1$ is the unit vector in the $x_1$ direction and
\begin{align}
\ol{U}^m(x_2) = -\frac{1}{2} + \frac{1}{2} \chi\left(x_2; \frac{\delta}{6}, \frac{\delta}{5}\right) \quad \text{when } x_2 < 0 \quad \text{and} \quad \ol{U}^m(x_2) = \frac{1}{2} - \frac{1}{2} \chi\left(x_2; \frac{\delta}{6}, \frac{\delta}{5}\right) \quad \text{when } x_2 \geq 0. 
\label{mean velocity}
\end{align}

%In rest of the proof we will use two parameters:
%\begin{align}
%\delta = c_1 \nu^{\frac{2}{3}} \qquad \ell = c_2 \nu^{\frac{1}{2}}.
%\end{align}
%The parameter $\delta$ denote the boundary layer thickness and $\ell$ denotes the wavelength in the $x_1$ direction of the inclined convection roll. We choose the vector field $\bs{U}$ to be the sum of two divergence-free velocity fields

%\begin{align}
%\ol{U}^{m}(x_2) \coloneqq
%\begin{cases}
%\frac{1}{2} \qquad \text{if} \quad -\frac{1}{2}+\frac{\delta}{5} \leq x_2 \leq \frac{1}{2} - \frac{\delta}{5}, \\
%0 \qquad \text{if} \quad -\frac{1}{2} < x_2 \leq -\frac{1}{2} + \frac{\delta}{6}, \\
%0 \qquad \text{if} \quad  \frac{1}{2} - \frac{\delta}{6} \leq x_2 \leq \frac{1}{2}
%\end{cases}
%\end{align}

%For example, one can choose

We define the convection roll velocity field $\bs{U}^c: \Omega \to \mathbb{R}^2$ with the help of a streamfunction. We first define $\Psi^c: \Omega \to \mathbb{R}$ as
\begin{align}
\Psi^c(\bs{x}) \coloneqq - \mathcal{A} \; \Xi(x_2) \; \frac{\sin(k(x_1 + x_2))}{k}, \quad \text{where} \quad 
\Xi(x) \coloneq \chi\left(x; \frac{\delta}{2}, {\delta}\right), \quad \text{and} \quad k \in \frac{2 \pi}{L_1} \mathbb{N}.
\end{align}
Using this streamfunction, we define
\begin{align}
\bs{U}^c & \coloneq \nabla^\perp \Psi^c \nonumber \\
& = \left( \mathcal{A}  \cos(k(x_1 + x_2)) \; \Xi(x_2) +  \mathcal{A} k^{-1} \sin(k(x_1 + x_2)) \; \Xi^\prime(x_2), \; -  \mathcal{A}  \cos(k(x_1 + x_2)) \; \Xi(x_2)\right),
\label{def Uc}
\end{align}
where $\nabla^\perp = (-\partial_{x_2}, \partial_{x_1})$.
A simple computation shows that the following estimates hold:
\begin{align}
 C_{l} \frac{\nu}{\delta} \leq \nu \langle |\nabla \bs{U}^m|^2 \rangle \leq  C_{u} \frac{\nu}{\delta}.
 \label{est Um}
\end{align}
and 
\begin{align}
\nu \langle |\nabla \bs{U}^c|^2 \rangle \leq C_c \nu \left( \mathcal{A}^2 k^2 + \frac{\mathcal{A}^2}{\delta} + \frac{\mathcal{A}^2}{k^2 \delta^3}\right), 
\label{est Uc}
\end{align}
for some positive constants $C_l$, $C_u$ and $C_c$ independent of any parameters and $\nu$. 
Finally, noting that $\supp \nabla \bs{U}^m \cap \supp \nabla \bs{U}^c = \emptyset$, we have $\nu \langle |\nabla \bs{U}|^2 \rangle =  \nu \langle |\nabla \bs{U}^m|^2 \rangle +  \nu \langle |\nabla \bs{U}^c|^2  \rangle$. Therefore, using (\ref{est Um}) and (\ref{est Uc}), we get the following upper and lower bound:
\begin{align}
\frac{C_{l}}{2} \frac{\nu}{\delta} \leq \nu \langle |\nabla \bs{U}|^2 \rangle \leq C \nu \left(\frac{1}{\delta} + \mathcal{A}^2 k^2 + \frac{\mathcal{A}^2}{\delta} + \frac{\mathcal{A}^2}{k^2 \delta^3}\right), 
\label{upper lower bound dissp U}
\end{align}
where $C$ is a constant independent of any parameter and the viscosity $\nu$.

%In our choices of parameters we will ensure
%\begin{align}
%C_c \nu \left( a_1^2 n^4 + \frac{a_1^2 n^2}{\delta} + \frac{a_1^2}{\delta^3}\right) \leq  \frac{C_{l}}{20} \frac{\nu}{\delta}
%\label{an imp cond tbs}
%\end{align}

\subsection{Choice of the velocity field $\wt{\bs{v}}$}
Once we have made our choice of the vector field $\bs{U}$, we want to obtain a lower bound on the dissipation in the vector field $\bs{u}$ using the variational principle (\ref{var prin: variational principle}). We define a divergence-free test velocity field $\wt{\bs{v}}$ which also consists of convection rolls. We first define a streamfunction $\wt{\psi}: \Omega \to \mathbb{R}$ as
\begin{align}
\wt{\psi}(\bs{x}) \coloneqq - \mathfrak{a}  \; \xi(x_2) \; \frac{\sin(k(x_1 + x_2))}{k}, \qquad \text{where} \qquad
\xi(x) = \chi\left(x; \frac{\delta}{4}, \frac{\delta}{3}\right).
\end{align}
Using this streamfunction, we define the velocity field $\wt{\bs{v}}$ as
\begin{align}
\wt{\bs{v}} & \coloneq \nabla^\perp \wt{\psi} \nonumber \\
& = \left( \mathfrak{a}  \cos(k(x_1 + x_2)) \; \xi(x_2) +  \mathfrak{a} k^{-1} \sin(k(x_1 + x_2)) \; \xi^\prime(x_2), \; -  \mathfrak{a}  \cos(k(x_1 + x_2)) \; \xi(x_2)\right).
\label{def v tilde}
\end{align}
We note that for our choices $\supp (\ol{U}^m)^\prime$, $\supp \Xi^\prime$ and $\supp \xi^\prime$ are pairwise disjoint. Consequently, we have
\begin{align}
\bs{U}^m = \bs{0} \qquad \text{when} \quad \bs{x} \in \supp \wt{\bs{v}}
\label{Um in supp v tilde}
\end{align}
and
\begin{align}
\wt{\bs{v}}  = \left( \mathfrak{a} \cos(k(x_1 + x_2)), \; -  \mathfrak{a} \cos(k(x_1 + x_2)) \right) \qquad \text{when} \quad \bs{x} \in \supp \bs{U}^c.
\label{v tilde in supp of Uc}
\end{align}

\subsection{Estimate on terms $I$, $II$ and $III$}
With definitions of $\bs{U}$ and $\wt{\bs{v}}$ in hand, we begin to estimate various terms in the lower bound (\ref{proof main thm: unknown lower bound}). The estimates on $I$ and $II$ are straightforward:
\begin{align}
I & = - 2 \langle U_2 \wt{v}_1 \rangle  = 2 \mathfrak{a} \mathcal{A}  \, \langle \, \cos^2(k(x_1+x_2)) \Xi(x_2) \, \rangle 
\end{align}
Noting that $\delta < \frac{1}{4}$, we get
\begin{align}
I \geq 2 \mathfrak{a} \mathcal{A}  \, \frac{1}{L} \int_{-\frac{1}{4}}^{\frac{1}{4}} \int_0^L \cos^2(k(x_1+x_2)) \, {\rm d} x_1 \, {\rm d} x_2 = \frac{\mathfrak{a} \mathcal{A}}{2}. 
\label{est I}
\end{align}
The estimate on the term $II$ is also a direct computation.
\begin{align}
II = - \nu \langle |\nabla \wt{\bs{v}}|^2 \rangle \geq - C_2 \nu (\mathfrak{a}^2 k^2 + \mathfrak{a}^2 \delta^{-1} + \mathfrak{a}^2 k^{-2} \delta^{-3}).
\label{est II}
\end{align}
The estimate on the $III$ term is simple as well but requires a few computations. From the expression of $\bs{U}$ and $\wt{\bs{v}}$, we note that
\begin{align}
\bs{U} \cdot \nabla \wt{\bs{v}} = \bs{U}^m \cdot \nabla \wt{\bs{v}} + \bs{U}^c \cdot \nabla \wt{\bs{v}}
\end{align}
Noting (\ref{def Um}) and the information about the support (\ref{Um in supp v tilde}), we get
\begin{align}
\mathbb{P} (\bs{U}^m \cdot \nabla \wt{\bs{v}}) = \bs{0}.
\label{eqn: Um grad v tilde}
\end{align}
Next, we calculate the term $\bs{U}^c \cdot \nabla \wt{\bs{v}}$. Using the information about the support (\ref{v tilde in supp of Uc}), we obtain
\begin{align}
\bs{U}^c \cdot \nabla \wt{\bs{v}} & =
\bs{U}^c \cdot \nabla
\begin{pmatrix}
\mathfrak{a}  \cos(k(x_1 + x_2)) \\
-  \mathfrak{a} \cos(k(x_1 + x_2)
\end{pmatrix} \nonumber \\
& = \mathcal{A} k^{-1} \sin(k(x_1 + x_2)) \; \Xi^\prime(x_2)  \partial_1 \begin{pmatrix}
\mathfrak{a} \cos(k(x_1 + x_2)) \\
-  \mathfrak{a} \cos(k(x_1 + x_2)
\end{pmatrix} \nonumber \\
& = \begin{pmatrix}
- \mathfrak{a} \mathcal{A} \, \Xi^\prime(x_2) \sin^2(k(x_1 + x_2)) \\
\mathfrak{a} \mathcal{A} \, \Xi^\prime(x_2) \sin^2(k(x_1 + x_2))
\end{pmatrix} \nonumber \\
& = \div M,
\label{eqn: Uc grad v tilde div M}
\end{align}
where
\begin{align}
M(x_1, x_2) = 
\begin{pmatrix}
0 & -F(x_1, x_2) \\
0 & F(x_1, x_2)
\end{pmatrix},
\quad
F(x_1, x_2) = \int_{0}^{x_2} \mathfrak{a} \mathcal{A} \, \Xi^\prime(y) \sin^2(k(x_1 + y)) \, {\rm d} y.
\end{align}
We note that $\supp \Xi^\prime \subseteq \left[-\frac{1}{2}, \, -\frac{1}{2} + \delta \right] \cup \left[\frac{1}{2} - \delta, \, \frac{1}{2} \right]$ and $|\Xi^\prime| \lesssim \delta^{-1}$, which leads to 
\begin{align}
\supp F \subseteq \left[-\frac{1}{2}, \, -\frac{1}{2} + \delta \right] \cup \left[\frac{1}{2} - \delta, \, \frac{1}{2} \right], \qquad \text{and} \qquad |F| \lesssim \mathfrak{a} \mathcal{A}.
\end{align}
Next, using (\ref{eqn: Um grad v tilde}) and (\ref{eqn: Uc grad v tilde div M}), we can write
\begin{align}
\nabla \Delta^{-1} \mathbb{P} \bs{U} \cdot \nabla \wt{\bs{v}} = \nabla \Delta^{-1} \mathbb{P} \div M.
\end{align}
Noting that $\nabla \Delta^{-1} \mathbb{P} \div$ is a Calder\'on--Zygmund operator of zeroth order, we obtain the following upper bound
\begin{align}
\langle |\nabla \Delta^{-1} \mathbb{P} \bs{U} \cdot \nabla \wt{\bs{v}}|^2 \rangle \lesssim \langle M^2 \rangle \lesssim \langle F^2 \rangle \leq  C_3 \mathfrak{a}^2 \mathcal{A}^2 \delta,
\end{align}
which leads to the following lower bound on $III$:
\begin{align}
III \geq  - \frac{C_3}{\nu} \mathfrak{a}^2 \mathcal{A}^2 \delta.
\label{est III}
\end{align}

%\begin{align}
%\langle |\nabla \Delta^{-1} \mathbb{P} \bs{U} \cdot \nabla \wt{\bs{v}}|^2 \rangle \leq C \frac{k_1^2 k_2^2 n^2}{\delta}.
%\end{align}

%\begin{align}
%\bs{U}^p \cdot \nabla \wt{\bs{v}} & =
%\bs{U}^p \cdot \nabla
%\begin{pmatrix}
%k_2 n \cos(n(x_1 + x_2)) \\
%-  k_2 n \cos(n(x_1 + x_2)
%\end{pmatrix} \nonumber \\
%& = k_1 \sin(n(x_1 + x_2)) \; \xi^\prime(x_2)  \partial_1 \begin{pmatrix}
%k_2 n \cos(n(x_1 + x_2)) \\
%-  k_2 n \cos(n(x_1 + x_2)
%\end{pmatrix} \nonumber \\
%& = \begin{pmatrix}
%-k_1 k_2 n^2 \xi^\prime(x_2) \sin^2(n(x_1 + x_2)) \\
%k_1 k_2 n^2 \xi^\prime(x_2) \sin^2(n(x_1 + x_2))
%\end{pmatrix} \nonumber \\
%& = \frac{1}{2}\begin{pmatrix}
%k_1 k_2 n^2 \xi^\prime(x_2) \cos(2n(x_1 + x_2)) \\
%-k_1 k_2 n^2 \xi^\prime(x_2) \cos(2n(x_1 + x_2))
%\end{pmatrix}
%+ \frac{1}{2}\begin{pmatrix}
%- k_1 k_2 n^2 \xi^\prime(x_2) \\
%k_1 k_2 n^2 \xi^\prime(x_2)
%\end{pmatrix} \nonumber \\
%= \div M_1 + \div M_2,
%\end{align}
%where
%\begin{align}
%M_1 = \frac{1}{4}\begin{pmatrix}
%k_1 k_2 n \xi^\prime(x_2) \sin(2n(x_1 + x_2)) & 0 \\
%-k_1 k_2 n \xi^\prime(x_2) \sin(2n(x_1 + x_2)) & 0
%\end{pmatrix} \quad \text{and} \quad 
%M_2 = \frac{1}{2}\begin{pmatrix}
%0 & k_1 k_2 n^2 ( 1 - \xi(x_2)) \\
%0 & - k_1 k_2 n^2 (1 - \xi(x_2))
%\end{pmatrix}
%\end{align}

\subsection{Putting everything together}
Combining the lower bounds on $I$, $II$ and $III$ from (\ref{est I}), (\ref{est II}) and (\ref{est III}) respectively, we obtain
\begin{align}
\nu \langle |\nabla \bs{u}|^2 \rangle - \nu \geq \frac{\mathfrak{a} \mathcal{A}}{2} - C_2 \nu (\mathfrak{a}^2 k^2 + \mathfrak{a}^2  \delta^{-1} + \mathfrak{a}^2 k^{-2} \delta^{-3}) - \frac{C_3}{\nu} \mathfrak{a}^2 \mathcal{A}^2  \delta.
\label{dissp u final est}
\end{align}
We make the following choices
\begin{align}
\mathcal{A} = \nu^{\frac{1}{6}}, \quad \mathfrak{a} = \mathfrak{c} \nu^{\frac{1}{6}}, \quad  \delta = \nu^{\frac{2}{3}}, \quad \nu^{-\frac{1}{2}} < k \leq \nu^{-\frac{1}{2}} + \frac{2 \pi}{L_1},
\label{choice of parameters}
\end{align}
such that $k \in \frac{2 \pi}{L_1} \mathbb{N} $. 
With these choices, both upper and lower bound
on $\nu \langle |\nabla \bs{U}|^2 \rangle$ given in (\ref{upper lower bound dissp U}) scales as $\nu^{1/3}$. Therefore, $\nu \langle |\nabla \bs{U}|^2 \rangle \sim \nu^{1/3}$. The positive constant $\mathfrak{c}$ is chosen as follows. We see that all the terms in (\ref{dissp u final est}) scale as $\nu^{1/3}$. Noting that the first term on the right-hand side in (\ref{dissp u final est}) is linear in $\mathfrak{c}$, while the last two are quadratic in $\mathfrak{c}$, we choose $\mathfrak{c}$ to be small enough (depending on $C_2$ and $C_3$) such that
\begin{align}
\nu \langle |\nabla \bs{u}|^2 \rangle \geq \frac{\mathfrak{c}}{4} \nu^{\frac{1}{3}} + \nu,
\end{align}
which is the required bound. 
Finally, we insisted $\delta < 1/4$ as a part of our construction. Therefore, we need $\nu < 1/8$ for our construction to work.

%With these choices, we also have that 
%\begin{align}
%\nu \langle |\nabla \bs{U}|^2 \rangle \sim \nu^{\frac{1}{3}}.
%\end{align}
%The positive constants $c_1$ and $c_2$ are chosen as follows. Recall, that the lower bound on $\nu \langle |\nabla \bs{U}|^2 \rangle$ given in (\ref{upper lower bound dissp U}) is contingent upon condition (\ref{an imp cond tbs}) is satisfied. With choice of parameters as in (\ref{choice of parameters}), both right-hand and left-hand side in (\ref{an imp cond tbs}) scales as $\nu^{1/3}$. Therefore, we simply choose $c_1$ small enough such that (\ref{an imp cond tbs}) holds.

%\subsection{Alternate way}
%Our next goal is to find a pressure $p$ such that
%\begin{align}
%-\nabla p = \bs{U}^p \cdot \nabla \wt{\bs{v}}
%\end{align}
%which means
%\begin{align}
%- \Delta p = \div (\bs{U}^p \cdot \nabla \wt{\bs{v}}) = k_1 k_2 n^2 \xi^{\prime \prime}(x_2) \sin^2(n(x_1 + x_2)) = \frac{k_1 k_2 n^2}{2} \xi^{\prime \prime} -  \frac{k_1 k_2 n^2}{2} \xi^{\prime \prime}  \cos(2 n(x_1 + x_2)).
%\end{align}
%A solution for the pressure is therefore given by
%\begin{align}
%p = p_1 + p_2,
%\end{align}
%where
%\begin{align}
%p_1 = - \frac{k_1 k_2 n^2}{2} \xi.
%\end{align}
%and 
%\begin{align}
%p_2 = \frac{k_1 k_2  n}{4 \pi} \int_{-\frac{1}{2}}^{x_2} \xi^{\prime \prime}(y) \sinh \left(2 n (x_2 - y) \right) \cos \left(2 n (x_1 + y)\right) \, {\rm d} y.
%\end{align}
\section{Branching Flows Based Design}
\label{sec: branching flows}
In this section, we provide the proof of the main Theorem \ref{thm: main thm}.  We present an example of convecting vector $\bs{U}$ with energy dissipation scaling as $(\log \nu^{-1})^{-2} $ such that the energy dissipation in $\bs{u}$ scales at least as $(\log \nu^{-1})^{-2} $. The strategy of the proof is similar to that of the previous section; we first identify a candidate design for $\bs{U}$ followed an application of the of variational principle (\ref{var prin: variational principle}). The design of $\bs{U}$  consists of a mean velocity $\bs{U}^m$ (similar to before) plus $\bs{U}^b$ a branching flow inclined at $3 \pi/4$ from the $x_1$ axis as depicted in Figure \ref{fig: branching}. In the next subsection, we set up the parameters required for the branching construction.

 \begin{figure}
\centering
 \includegraphics[scale = 0.5]{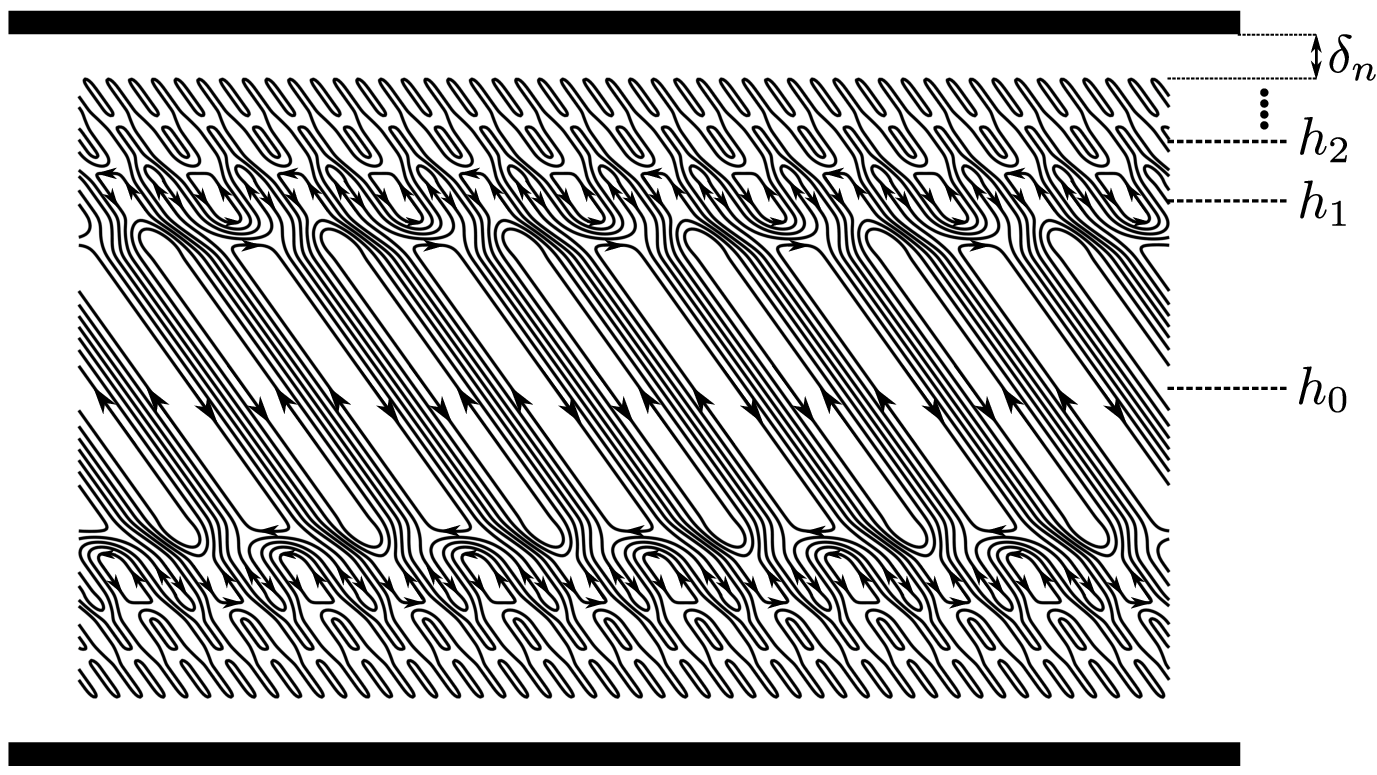}
 \caption{shows the schematic of the branching flow $\bs{U}^b$.}
 \label{fig: branching}
\end{figure}

\subsection{Parameters for the branching construction}
\label{subsec: para branching}
The branching flows consist of several horizontal layers, with the period of the flow doubling from one layer to the next. Within each layer, the velocity field predominantly flows at an angle of $3 \pi/4$. See Figure \ref{fig: branching} for a sketch. We denote the thickness of the $i$th boundary layer by $\delta_i$ for $i \in \{0, \dots n\}$. The relationship between the thicknesses of different boundary layers is governed by
\begin{align}
\delta_i = 2^{-2} \delta_{i-1}, \quad \implies \quad \delta_i = 2^{- 2 i } \delta_{0} \quad \text{for all} \quad i \in \{1, \dots n\}.
\label{branching: bl thick relation}
\end{align}
Here, $\delta_0$ is the thickness of the first layer. The reason for the $2^{-2i}$ decrease in the boundary layer thicknesses will become apparent in later analysis. For the layers lying above and including the midplane ($x_2 = 0$), the $x_2$ coordinates of the start of the layers are given by
\begin{align}
h_0 = 0, \qquad h_i = \sum_{j = 0}^{i-1} \delta_j \quad \text{for } i \in \{1, \dots, n + 1\}.
\end{align}
Our branching flow construction is mirror symmetric about the midplane. Therefore, the $x_2$ coordinates of the starts of the layers below the midplane simply given by $-h_i$. We ensure that the sum of the thicknesses of all these layers (above the midplane) adds up to half the channel width, namely $1/2$, which then leads to
\begin{align}
\sum_{i = 0}^{n} \delta_i = \frac{1}{2}, \quad \implies \quad \delta_0 \frac{1 - 2^{-2(n+1)}}{1 - 2^{-2}} = \frac{1}{2} \quad \implies \quad \frac{1 - 2^{-2}}{2} \leq \delta_0 < \frac{1}{2}.
\label{branching: half channel width}
\end{align}
Consequently, $\delta_0$ is $O(1)$, chosen such that (\ref{branching: half channel width}) is satisfied. Next, we denote the typical horizontal wavenumber of the branching flow at the beginning of the $i$th layer by $k_i$ for $i \in \{0, \dots n\}$. These wavenumbers obey the following relation
\begin{align}
k_i = 2 k_{i-1}, \quad \implies \quad k_i = 2^i k_0,
\label{branching: x wavnum}
\end{align}
where $k_0 \in \frac{2 \pi}{L} \mathbb{N}$ is the wavenumber at $x_2 = 0$ and will be chosen as part of the proof. In our proof, we will ensure that $\delta_i > k_i^{-1}$ for $i \in \{0, \dots n\}$, which only requires
\begin{align}
\delta_n > k_n^{-1}. 
\label{branching: extra cond 1}
\end{align}

We now define a family of smooth cutoff functions $\zeta_i:[-1/2, 1/2] \to [0, 1]$ for $i \in \{0, \dots n\}$, which satisfy the following properties:
\begin{enumerate}[label = (\roman*)]
\item 
$\supp \zeta_0 \subseteq [-h_1, h_1]$ and 
$
\supp \zeta_i \subseteq [-h_{i+1}, -h_{i-1}] \cup [h_{i-1}, h_{i+1}]
$ for $1 \leq i \leq n$,
\item
$
|\nabla^j \zeta_i| \leq C_j \delta_i^{-j},
$
\item 
$
\sum_{i = 0}^{n} \zeta_i^2 \leq 1 \qquad \text{with} \quad \sum_{i = 0}^{n} \zeta_i^2 = 1 \quad \text{when} \quad x_2 \in \left[-\frac{1}{2} + \delta_n, \frac{1}{2} - \delta_n\right].
$
\end{enumerate}
One can indeed find such a family of cutoff function. For example, following \cite{doering2019optimal}, we can choose
\begin{align}
\zeta_i(x_2) = f_i(x_2) + f_{i-1}(\delta_i + 2 h_i -x_2) \quad \text{for} \quad x_2 \in [0, 1/2] \quad \text{and} \quad \zeta_i(x_2) = \zeta_i(-x_2) \quad \text{for} \quad x_2 \in [-1/2, 0],
\end{align}
where 
\begin{align}
\label{functions f and fn}
& f(x_2) = \sqrt{\frac{1}{2} - \frac{1}{2} \tanh\left(\frac{x_2 - \frac{1}{2}}{x_2^2(1-x_2)^2}\right)}, \\ \nonumber
& f_i(x_2) = f\left(\frac{x_2 - h_i}{\delta_i}\right) \text{ for } i \in \{0, \dots n - 1\}, \quad f_n(x_2) = f\left(\frac{x_2 - h_n}{(\delta_n/2)}\right) \text{ and } f_{-1} \equiv 0.
\end{align}
The item (iii) holds because the function $f$ has the properties: $f^2 \leq 1$ and $f^2(x_2) + f^2(1-x_2) = 1$.

 \begin{figure}
\centering
 \includegraphics[scale = 0.5]{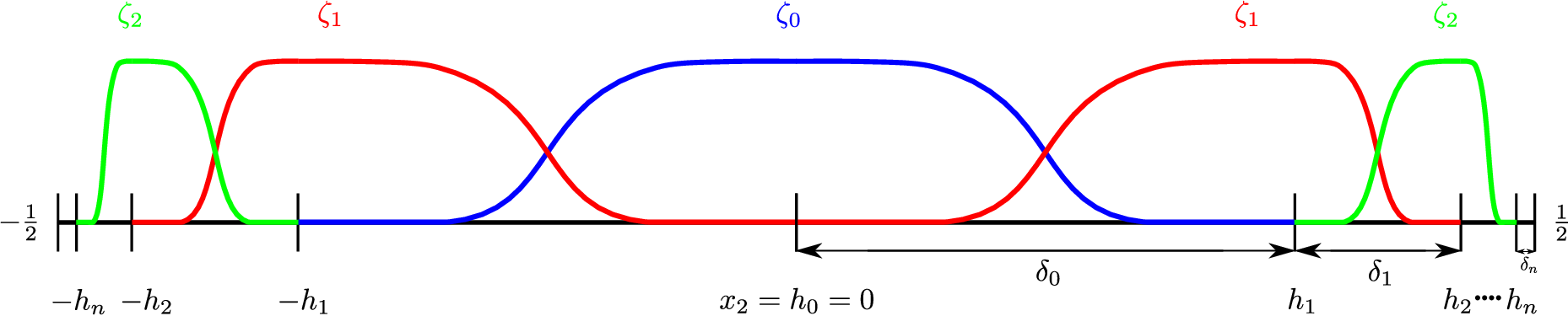}
 \caption{shows cutoff functions $\zeta_i$.}
 \label{fig: parameters}
\end{figure}

\subsection{Choice of the velocity field $\bs{U}$ and $\wt{\bs{v}}$}
We write the velocity field $\bs{U}$ as a sum of two divergence-free velocity fields $\bs{U} = \bs{U}^m + \bs{U}^b$. We define the mean velocity $\bs{U}^m$ as in (\ref{mean velocity})
with $\delta$ replaced with $\delta_n$. We define $\bs{U}^b$ and $\wt{\bs{v}}$ based on streamfunctions as
\begin{align}
&\Psi(\bs{x}) \coloneqq - \sum_{i=0}^{n} \mathcal{A} \; \zeta_i(x_2) \; \frac{ \sin(k_i(x_1 + x_2))}{k_i}, \qquad   \bs{U}_b \coloneqq \nabla^\perp \Psi, \\
&\wt{\psi}(\bs{x}) \coloneqq - \sum_{i=0}^{n} \mathfrak{a} \; \zeta_i(x_2) \; \frac{ \sin(k_i(x_1 + x_2))}{k_i}, \qquad \wt{\bs{v}} \coloneqq \nabla^\perp \wt{\psi}.
\end{align}
Therefore, the choice of $\bs{U}_b$ and $\wt{\bs{v}}$ only differ in the prefactors. From our choices it is clear that 
\begin{align}
\supp_z \bs{U}^m \subseteq \left[-\frac{1}{2}, -\frac{1}{2} + \frac{\delta_n}{5} \right) \cup \left(\frac{1}{2} - \frac{\delta_n}{5}, \frac{1}{2} \right] \text{ and } \supp_z \bs{U}^b, \supp_z \wt{\bs{v}} \subseteq \left[-\frac{1}{2} + \frac{\delta_n}{2}, \frac{1}{2} - \frac{\delta_n}{2} \right]. 
\label{sppt Um Ub vtilde}
\end{align}
The property about the support of $\bs{U}^b$ and $\wt{\bs{v}}$ is true because the function $f_n$ in (\ref{functions f and fn}) is defined based on the boundary layer thickness $\delta_n/2$. Next, we have the following estimate on $\nabla \bs{U}^b$:
\begin{align}
\nu \langle |\nabla \bs{U}^b|^2 \rangle \lesssim \nu \mathcal{A}^2 \sum_{i = 0}^{n} \left(k_i^2 \delta_i + \frac{1}{k_i^2 \delta_i^3}\right) \lesssim \nu \mathcal{A}^2 k_0^2 \delta_0 n,
\label{branching: est grad Ub}
\end{align}
where we used (\ref{branching: bl thick relation}), (\ref{branching: x wavnum}) and (\ref{branching: extra cond 1}) to obtain the last relation.
Now, we can find an estimate on $\nabla \bs{U}^m$ as in (\ref{est Um}) with $\delta$ replaced with $\delta_n$. Finally, noting the information about the support in (\ref{sppt Um Ub vtilde}), we get
\begin{align}
C_l \frac{\nu}{\delta_n} \leq \nu \langle |\nabla \bs{U}|^2 \rangle  \leq C \left( \frac{\nu}{\delta_n} + \nu \mathcal{A}^2 k_0^2 \delta_0 n \right),
\label{branching: up low bdd grad U}
\end{align}
for constants $C_l$ and $C$ independent of viscosity.

%Similar to before, we see that
%\begin{align}
%\nu \langle |\nabla \bs{U}|^2 \rangle = \nu \langle |\nabla \bs{U}^m|^2 \rangle + \nu \langle |\nabla \bs{U}^b|^2 \rangle \lesssim \frac{\nu}{\delta_n} + \nu \mathcal{A}^2 \sum_{i = 0}^{n} \left(k_i^2 \delta_i + \frac{1}{k_i^2 \delta_i^3}\right).
%\end{align}

\subsection{Estimate on terms $I$, $II$ and $III$}
In the calculation of $I = -2 \langle U_2 \wt{v}_1 \rangle$, only the interactions of the same modes in $U_2$ and $\wt{v}_1$ survive, which leads to
\begin{align}
I = - 2 \left\langle \mathfrak{a} \mathcal{A} \sum_{i = 0}^n \zeta_i^2 \cos^2(k_i (x_1 + x_2)) \right\rangle \geq \frac{\mathfrak{a} \mathcal{A}}{2}.
\label{branching: est I}
\end{align}
Here, we used that $\delta_n \leq 1/4$ and item (iii) about properties of $\zeta_i$ in subsection \ref{subsec: para branching}. The calculation of term $II$ is similar to (\ref{branching: est grad Ub}), which yields
\begin{align}
II = - \nu \langle |\nabla \wt{\bs{v}}|^2 \rangle \geq - C_2 \mathfrak{a}^2 k_0^2 \delta_0 n.
\label{branching: est II}
\end{align}
Estimate on term $III$ requires a few steps. We first write down
\begin{align}
\bs{U} \cdot \nabla \wt{\bs{v}}  = \ol{\bs{W}} + \wt{\bs{W}},
\end{align}
where we define the components of $\ol{\bs{W}}$ and $\wt{\bs{W}}$ as follows
\begin{align}
(\bs{U} \cdot \nabla \wt{\bs{v}})_1 & = \ol{W}_1 + \wt{W}_1 \nonumber \\
& \coloneqq - \mathfrak{a} \mathcal{A} \left(\zeta_i \zeta_i^\prime + \zeta_{i+1} \zeta_{i+1}^\prime \right) + \left(\sum_{j = 1}^{4} \alpha^c_j \cos(j k_i(x_1 + x_2)) + \sum_{j = 1}^{4} \alpha^s_j \sin(j k_i(x_1 + x_2)) \right),
\label{x1 comp of U grad v} \\[5pt]
(\bs{U} \cdot \nabla \wt{\bs{v}})_2 & = \ol{W}_2 + \wt{W}_2 \coloneqq  \mathfrak{a} \mathcal{A} \left(\zeta_i \zeta_i^\prime + \zeta_{i+1} \zeta_{i+1}^\prime \right) - \sum_{j = 1}^{4} \alpha^c_j \cos(j k_i(x_1 + x_2)). 
\label{x2 comp of U grad v}
\end{align}
We note the coefficients $\alpha^c_j$ and $\alpha^s_j$ in (\ref{x1 comp of U grad v}) and (\ref{x2 comp of U grad v}) are smooth function of $x_2$ and are given by
\begin{align}
& \alpha^c_1 = - \frac{3}{4} \mathfrak{a} \mathcal{A} (2 \zeta_{i+1} \zeta_i^\prime + \zeta_i \zeta_{i+1}^\prime), \qquad
\alpha^c_2 = 0, \qquad
\alpha^c_3 = \frac{1}{4} \mathfrak{a} \mathcal{A} (2 \zeta_{i+1} \zeta_i^\prime - \zeta_i \zeta_{i+1}^\prime), \qquad
\alpha^c_4 = 0, \\
& \alpha^s_1 = - \frac{1}{4 k_i} \mathfrak{a} \mathcal{A} \left(\zeta_i^\prime \zeta_{i+1}^\prime - 2 \zeta_{i+1} \zeta_i^{\prime \prime} + \zeta_i \zeta_{i+1}^{\prime \prime}\right), \qquad \alpha^s_2 = \frac{1}{2 k_i} \mathfrak{a} \mathcal{A} \left((\zeta_i^\prime)^2 - \zeta_i \zeta_i^{\prime \prime}\right), \\
& \alpha^s_3 = \frac{1}{4 k_i} \mathfrak{a} \mathcal{A} \left(3 \zeta_i^\prime \zeta_{i+1}^\prime -  \zeta_{i+1} \zeta_i^{\prime \prime} - \zeta_i \zeta_{i+1}^{\prime \prime}\right), \qquad \alpha^s_4 = \frac{1}{4 k_i} \mathfrak{a} \mathcal{A} \left((\zeta_{i+1}^\prime)^2 - \zeta_{i+1} \zeta_{i+1}^{\prime \prime}\right).
\end{align}
Enforcing the condition (\ref{branching: extra cond 1}), we note one important estimate on these coefficients:
\begin{align}
\|\alpha^c_j\|_{L^\infty}, \|\alpha^s_j\|_{L^\infty}, \leq c \mathfrak{a} \mathcal{A} \delta_{i}^{-1}, 
\end{align}
for some constant $c$ independent of $j$ and $\nu$. Now, an estimate on term $III$ can be written as 
\begin{align}
|III| = \nu^{-1} \langle |\nabla \Delta^{-1} \mathbb{P} \bs{U} \cdot \nabla \wt{\bs{v}}|^2 \rangle \leq 2 \nu^{-1} \langle |\nabla \Delta^{-1} \mathbb{P} \ol{\bs{W}}|^2 \rangle + 2 \nu^{-1} \langle |\nabla \Delta^{-1} \mathbb{P} \wt{\bs{W}}|^2 \rangle \eqcolon E_1 + E_2.
\label{bound on term III 1}
\end{align}
We first focus on estimating $E_1$. From (\ref{x1 comp of U grad v}) and (\ref{x2 comp of U grad v}), we notice that 
\begin{align}
\ol{\bs{W}} = \left(-\frac{\mathfrak{a} \mathcal{A}}{2}, \frac{\mathfrak{a} \mathcal{A}}{2} \right)^{T} \sum_{i = 0}^{n} (\zeta_i^2)^\prime.
\end{align}
Next, we observe that for a smooth function $f(x_1, x_2) \coloneq f(x_2)$, which satisfies $f(-1/2) = f(1/2) = 0$, we have $\mathbb{P} (0, \partial_{x_2} f) = 0$. Therefore, we can write  $$\nu^{-1} \langle |\nabla \Delta^{-1} \mathbb{P} \ol{\bs{W}}|^2 \rangle = \nu^{-1} \langle |\nabla \Delta^{-1} \ol{W}_1|^2 \rangle.$$ This is because $\mathbb{P}(\ol{W}_1, 0) = (\ol{W}_1, 0)$ as $\ol{W}_1$ is only a function of $x_2$. To estimate $\nu^{-1} \langle |\nabla \Delta^{-1} \ol{W}_1|^2 \rangle$, we note the following lemma.
\begin{lemma}
Let $f(x_1, x_2) \coloneqq f(x_2)$ be a nonnegative smooth function such that $f(x_2) \leq 1$ and $f(x_2) = 1$ when $x_2 \in [-1/2 + \delta, 1/2 - \delta]$ for some $\delta \leq 1/4$. Then 
\begin{align}
\dashint_{\Omega} |\nabla \Delta^{-1} \partial_{x_2} f|^2 \, {\rm d} \bs{x} \leq 6 \delta.
\label{div inverse func x2 bound}
\end{align}
\label{div inverse func x2}
\end{lemma}
\begin{proof}
We note that 
\begin{align}
\Delta^{-1} \partial_{x_2} f = \int_{-\frac{1}{2}}^{x_2} f(x_2^\prime) \, {\rm d} x_2^\prime - \left(x_2 + \frac{1}{2} \right) \int_{-\frac{1}{2}}^{\frac{1}{2}} f(x_2^\prime) \, {\rm d} x_2^\prime,
\end{align}
which implies
\begin{align}
\nabla \Delta^{-1} \partial_{x_2} f = f(x_2) - \int_{-\frac{1}{2}}^{\frac{1}{2}} f(x_2^\prime) \, {\rm d} x_2^\prime.
\end{align}
Therefore, we have
\begin{align}
|\nabla \Delta^{-1} \partial_{x_2} f| \leq |1 - f(x_2)| + \left|1 - \int_{-\frac{1}{2}}^{\frac{1}{2}} f(x_2^\prime) \, {\rm d} x_2^\prime\right| \leq |1 - f(x_2)|  + 2 \delta.
\end{align}
Therefore, we get the following bound
\begin{align}
|\nabla \Delta^{-1} \partial_{x_2} f| \leq
\begin{cases}
1 + 2 \delta \quad &\text{when} \quad  x_2 \left(-\frac{1}{2}, - \frac{1}{2} + \delta \right) \cup \left(\frac{1}{2} + \delta,  \frac{1}{2} \right), \\
2 \delta \quad &\text{when} \quad  x_2 \in \left[-\frac{1}{2} + \delta,  \frac{1}{2} - \delta \right].
\end{cases}
\end{align}
Finally, a volume integration gives the required bound (\ref{div inverse func x2 bound}).
\end{proof}
Since, we have $\ol{W}_1 = - \partial_{x_2} ( \mathfrak{a} \mathcal{A} \sum_{i=0}^{n} \zeta_i^2/2)$, using Lemma \ref{div inverse func x2}, we obtain the following bound:
\begin{align}
E_1 \lesssim \nu^{-1} \mathfrak{a}^2 \mathcal{A}^2 \delta_n.
\label{branching: est E1}
\end{align}
Next, we estimate $E_2$ in (\ref{bound on term III 1}) as follows:
\begin{align}
2 \nu^{-1} \langle |\nabla \Delta^{-1} \mathbb{P} \wt{\bs{W}}|^2 \rangle & = 2 \nu^{-1} \left\langle \left|\nabla \Delta^{-1} \mathbb{P} \partial_{x_1} \scaleobj{.8}{\int_{0}^{x_1}}  \wt{\bs{W}}(x_1^\prime, x_2) \, {\rm d} x_1^\prime \right|^2 \right\rangle \nonumber \\
& \lesssim  \nu^{-1} \left\langle \left| \scaleobj{.8}{\int_{0}^{x_1}}  \wt{\bs{W}}(x_1^\prime, x_2) \, {\rm d} x_1^\prime\right|^2 \right\rangle \nonumber \\
& \lesssim \nu^{-1}  \mathfrak{a}^2 \mathcal{A}^2 \sum_{i = 0}^{n} \frac{1}{k_i^2 \delta_i} \nonumber \\
& \lesssim C \nu^{-1}  \mathfrak{a}^2 \mathcal{A}^2 \frac{n}{k_0^2 \delta_0}.
\label{branching: est E2}
\end{align}
The function $\scaleobj{.8}{\int_{0}^{x_1}}  \wt{\bs{W}}(x_1^\prime, x_2) \, {\rm d} x_1^\prime$ is periodic in $x_1$ and is therefore well-defined on $\Omega$. To obtain the second line in (\ref{branching: est E2}), we used the fact that $\nabla \Delta^{-1} \mathbb{P} \partial_{x_1}$ is a Calder\'on--Zygmund operator of zeroth-order. To obtain the last line, we note that
\begin{align}
\left| \scaleobj{.8}{\int_{0}^{x_1}}  \wt{\bs{W}}(x_1^\prime, x_2) \, {\rm d} x_1^\prime\right| \lesssim \frac{\mathfrak{a} \mathcal{A}}{k_i \delta_i}, \quad \text{when} \quad x_2 \in [-h_{i+1}, h_i] \cup [h_i, h_{i+1}].
\end{align}
Combining (\ref{branching: est E1}) and (\ref{branching: est E2}), we get an estimate on term $III$:
\begin{align}
|III|
& \leq C_3 \nu^{-1}  \mathfrak{a}^2 \mathcal{A}^2 \left( \frac{\delta_0}{2^{2 n}} +  \frac{n}{k_0^2 \delta_0} \right) 
\label{branching: est III}
\end{align}

\subsection{Proof of Theorem \ref{thm: main thm}}
We make the following choice of the parameters:
\begin{align}
&\qquad \qquad \qquad \mathcal{A} = \frac{1}{\log_2 \nu^{-1}}, \qquad \qquad \mathfrak{a} = \frac{\mathfrak{c}}{\log_2 \nu^{-1}}, \\[5pt]
& \frac{1}{\nu^{\frac{1}{2}} \left(\log_2 \nu^{-1}\right)^{\frac{1}{2}}} < k_0 \leq \frac{1}{\nu^{\frac{1}{2}} \left(\log_2 \nu^{-1}\right)^{\frac{1}{2}}} + \frac{2 \pi}{L}, \qquad n = \left\lfloor \frac{1}{2} \log_2 \left(\frac{1}{\nu \log_2^2 \nu^{-1}}\right) \right\rfloor.
\end{align}
With these choices, from (\ref{branching: up low bdd grad U}), we see that $\nu \langle |\nabla \bs{U}|^2 \rangle \sim (\log_2 \nu^{-1})^{-2}.$ Next, we see that the estimates (\ref{branching: est I}), (\ref{branching: est II}) and (\ref{branching: est III}) for terms $I$, $II$ and $III$, respectively, also scale as $ (\log_2 \nu^{-1})^{-2}.$ Therefore, by choosing $\mathfrak{c}$ to be small enough, we get that $\nu \langle |\nabla \bs{u}|^2 \rangle \gtrsim (\log_2 \nu^{-1})^{-2}.$ Finally, to satisfy the condition (\ref{branching: extra cond 1}), along with the constraint $\delta_n \leq 1/4$, we require $\nu < 1/500.$

%\begin{align}
%C_l \frac{\nu}{\delta_n} \leq \nu \langle |\nabla \bs{U}|^2 \rangle  \leq C_u \left( \frac{\nu}{\delta_n} + \nu \mathfrak{a}^2 k_0^2 \delta_0 n \right).
%\end{align}
%\begin{align}
%I = \langle U_2 \wt{v}_1 \rangle \geq C_1 \mathfrak{a} \mathcal{A} 
%\end{align}
%\begin{align}
%II = \nu \langle |\nabla \wt{\bs{v}}|^2 \rangle \leq C_2 \nu \mathfrak{a}^2 k_0^2 \delta_0 n
%\end{align}
%\begin{align}
%|III| \leq C_3 \nu^{-1}  \mathfrak{a}^2 \mathcal{A}^2 \left( \frac{\delta_0}{2^{2 n}} +  \frac{n}{k_0^2 \delta_0} \right) 
%\end{align}
\section{Conclusion and future outlook}
\label{sec: conclusion}
In this paper, we explored the phenomenon of anomalous and enhanced dissipation in a passive divergence-free vector field in a Couette flow setting. One of the main result of the paper is an upper bound on the rate of energy dissipation $\varepsilon_{\bs{u}}$ in the passive vector field $\bs{u}$ in terms of the energy dissipation $\varepsilon_{\bs{U}}$in the convecting vector field $\bs{U}$ (see Theorem \ref{thm: upper bound}). Complementary to this upper bound, in Theorem \ref{thm: main thm}, we provided a construction of a family of convecting vector fields $\{\bs{U}_{\nu}\}_{\nu > 0}$ dissipating energy as $(\log \nu^{-1})^2$ for which the energy dissipation in the passive vector $\bs{u}$ goes at least as $(\log \nu^{-1})^2$. The proof of this result is based on a variational principle for the rate of energy dissipation given in Section \ref{sec: var prin}. Our results in Theorem \ref{thm: upper bound} and Theorem \ref{thm: main thm} can be restated in terms of the momentum transport or the transverse current of stream-wise ($x_1$) component of velocity $\bs{u}$:
\begin{align}
J^{\bs{u}} = \nu \ol{\partial_2 u_1} - \ol{U_2 u_1},
\end{align}
as $J^{\bs{u}}$ relates to energy dissipation through formula (\ref{eqn: dissp for 3}). 

In the context of the incompressible vector diffusion equation, a key question of interest is Question \ref{Q: anom dissp div free} in the Couette flow setting. Most studies on anomalous dissipation in passive scalars rely on the Lagrangian viewpoint, which cannot be directly applied to the passive vector equation considered in this paper. Consequently, addressing this question requires novel approaches that go beyond traditional Lagrangian methods. The variational approach presented in this paper could be one such potential method of tackling this question. We used two-dimensional branching flows within the variational principle, which allows us to answer this question with a logarithmic scaling discrepancy. It is known from the studies of optimal scalar transport problems, that two-dimensional branching flows lead to a logarithmic scaling miss \citep{doering2019optimal}. In \citep{kumar2022three, kumar2023bulk}, it was shown that such logarithmic correction can be eliminated (thus achieving the clean scaling of optimal scalar transport) through the use of three-dimensional branching flows. Consequently, it is plausible that, in the context of passive vector fields as well, a construction based on three-dimensional branching flows could fully resolve Question \ref{Q: anom dissp div free}. Another promising approach that can be fruitful is due to \citet{armstrong2023anomalous} on anomalous dissipation in a passive scalar, which is based on homogenization techniques. It would be interesting to see whether their approach can be adapted to resolve this question.

We note that Question \ref{Q: anom dissp div free} may likely be answered by considering a freely decaying flow, i.e., a flow within a periodic domain or a smooth bounded domain with homogeneous Dirichlet boundary conditions, in the absence of forcing, starting from an initial condition $\bs{u}_0$ independent of viscosity. In this scenario, the long-time average should be replaced with a finite-time average. We can consider this problem in three dimensions and choose the velocity fields of the form $\bs{U} = (U_1(x_1, x_2), U_2(x_1, x_2), 0)$ and $\bs{u} = (0, 0, u_3(x_1, x_2))$. The governing equation for $u_3$ then becomes the advection-diffusion equation. Finally, to answer Question \ref{Q: anom dissp div free} one might consider choosing $\bs{U}$ as an alternating shear flow or a checkered board flow, as employed by \citep{drivas2022anomalous, colombo2023anomalous, elgindi2023norm} in studies of anomalous dissipation in a passive scalar. However, we do not pursue this solution in this paper because the velocity fields $\bs{U}$ and $\bs{u}$ differ significantly, which we think is an artifact of the finite-time average. Consequently, such a solution is unlikely to exist in the settings discussed in the Introduction, where there is a continuous supply of energy at the larger scales and a long-time average is used in the definition of energy dissipation, which renders the transient nature of the problem irrelevant.  But more importantly, our future interest lies in the potential to close the gap between $\bs{U}$ and $\bs{u}$ using computational tools, which is the primary motivation for this study. Because $\bs{U}$ and $\bs{u}$ described above are so distinct that it seems unlikely they could be unified to obtain a solution to the Navier--Stokes equation.  For example, such a solution does not appear to work when the number of input/output passes are increased to $n = 2$ (see Question \ref{Q: multiple input/output} below).

We conclude the paper by noting a few problems based on the input/output perspective taken in this paper, where we believe it is possible to make progress using both analysis and computational tools and can provide valuable insights into the solutions of Navier--Stokes equation. The first problem is related to  optimal mixing of a passive scalar in the limit of zero diffusion, which has been a focus of recent studies. For cases where the underlying vector field is Sobolev, $\bs{U} \in L^\infty_t W_{\bs{x}}^{1, p}$, this problem has been examined using both computational approaches \citep{lin2011optimal, iyer2014lower}  and rigorous upper \citep{crippa2008estimates, seis2013maximal, iyer2014lower} and lower bounds \citep{YaoZlatos17, AlbertiCrippaMazzucato19, ElgindiZlatosuniversalmixer}, demonstrating that the optimal mixing rate is exponential in time. However, one crucial drawback of these studies is that the velocity field lacks a governing equation. In particular, the velocity field does not arise as a solution of the Navier--Stokes equation in one of the three settings discussed in the Introduction. The point of view taken in this paper can remedy this gap to certain extent. In particular, we ask whether the exponential mixing still holds if the vector field $\bs{u}$ is a solution of (\ref{eqn: conv diff u}{\color{blue}a-b}) and convects a scalar $\theta$:
\begin{question}[Mixing of a scalar]
Consider the advection of a scalar $\theta$, $\partial_t \theta + \bs{u} \cdot \nabla \theta = 0$, by a divergence-free velocity field $\bs{u}$ which itself is obtained from solving $\partial_t \bs{u} + \bs{U} \cdot \nabla \bs{u} = - \nabla p + \nu \Delta \bs{u}$. The velocity field $\bs{U}$ is smooth divergence free and satisfies the bound $\norm{\bs{U}(t, \cdot)}_{W^{1, p}} \leq B$. Is it then true that for $s < 0$, $\norm{\theta(t, \cdot)}_{H^s} \geq \exp(-c B t)$ for some constant $c$ independent of $B$.
\end{question}
A complementary question would be if exponential mixing is a lower bound then is it sharp, i.e., can one produce examples of $\bs{U}$ for which the exponential mixing holds? A similar question can be asked regarding the mixing behavior of a passive vector field.
\begin{question}[Mixing of a vector]
Consider the advection of a vector of a divergence-free velocity field $\bs{u}$ which itself is obtained from solving $\partial_t \bs{u} + \bs{U} \cdot \nabla \bs{u} = - \nabla p$. The velocity field $\bs{U}$ is smooth divergence free and satisfies the bound $\norm{\bs{U}(t, \cdot)}_{W^{1, p}} \leq B$. Is it then true that for $s < 0$, $\norm{\bs{u}(t, \cdot)}_{H^s} \geq \exp(-c B t)$ for some constant $c$ independent of $B$.
\end{question}
The equation (\ref{eqn: conv diff u}) discussed in this paper represents a single level of input/output pass. However, one could more broadly consider the mixing questions mentioned above or the problem of anomalous dissipation within a framework involving multiple levels of input/output pass. For instance, one might ask
\begin{question}[Multiple input/output passes]
\label{Q: multiple input/output}
Consider a set of partial differential equations \begin{align}
\partial_t \bs{u}_i + \bs{u}_{i-1} \cdot \nabla \bs{u}_i = - \nabla p_i + \nu \Delta \bs{u}_i \qquad i \in \{1, \dots n\},  
\label{eqn: multiple input/output}
\end{align}
where the velocity field $\bs{u}_i$ for all $i \in \{0, \dots n\}$ are divergence-free. Is it possible to produce an example  of smooth velocity fields $\bs{u}_0$ which exhibit anomalous dissipation such that $\nu \langle | \nabla \bs{u}_n|^2 \rangle \geq c_n > 0,$ where $c_n$ is a constant independent of the viscosity. A more challenging question is whether  $c_n$ can also be made independent of $n$.
\end{question}

The fixed point iteration (\ref{eqn: multiple input/output}) (without the time derivative) is used to find numerical solution of the steady Navier--Stokes equation \citep{kay2002preconditioner} and goes by the name of Arrow–Hurwicz algorithm \citep{temam2024navier}.

%The set of equations (\ref{eqn: multiple input/output}) is nothing but represents a fixed-point iteration, commonly used, for example, to construct solutions of the Navier–Stokes equation. Consequently, studying the limiting behavior as $n \to \infty$ in the aforementioned question becomes particularly interesting. As such, a difficult but useful question is whether the constant $c_n$ can also be made independent of $n$. 

Next, we consider the Naiver--Stokes equation under the Boussinesq approximation coupled with the advection-diffusion equation. However, we replace the term $\bs{u} \cdot \nabla \bs{u}$ by $\bs{U} \cdot \nabla \bs{u}$. Consequently, the momentum part of the equation becomes linear but the interaction of $\bs{u}$ with $T$ in the convection-diffusion equation leads to nonlinearity. One can consider these equations, for example, in the standard Rayliegh--B\'enard setup. The Rayleigh--B\'enard convection (RBC) is the flow driven by buoyancy in a differentially heated layer of fluid. In particular, the velocity fields $\bs{U}$ and $\bs{u}$ satisfy homogeneous Dirichlet boundary condition and the temperature $T = 1$ at $x_2 = -1/2$ and $T = 0$ at $x_2 = 0$. The two nondimensional parameters governing the equations are the Rayleigh number (based on $1$ unit gap width and $1$ unit temperature difference) $Ra = \nu^{-1} \kappa^{-1} g \alpha$ and the Prandtl number $Pr = \kappa^{-1} \nu$, where $\nu$, $\kappa$, $\alpha$ and $g$ are the viscosity, thermal diffusivity, coefficient of thermal expansion and the gravitational acceleration, respectively. The Nusselt number, which quantifies nondimensional heat transfer,  is given by $Nu = \langle |\nabla T|^2 \rangle = 1 + \langle u_2 T \rangle = 1 + Ra^{-1} \langle |\nabla \bs{u}|^2\rangle$. One of the competing theories of heat transfer suggests that for Prandtl number order one, the dimensional rate of heat transfer becomes independent of the molecular diffusion (upto possibly a logarithmic correction) in the limit of large Rayleigh number, which in terms of Nusselt number is $Nu \sim Ra^{1/2}$ \citep{kraichnan1962turbulent, grossmann2000scaling}. This rate of heat transfer leads to anomalous dissipation in the velocity field $\bs{u}$.  Given this prediction, we pose the following question, which is presumably easier to answer than the full nonlinear system (the Naiver--Stokes equation under the Boussinesq approximation) but more challenging than a purely linear system.
\begin{question}[A nonlinear problem]
Consider the following set of equations
\begin{align}
& \partial_t \bs{u} + \bs{U} \cdot \nabla \bs{u} = - \nabla p + Pr \Delta \bs{u} + Pr Ra T \bs{e}_2, \\
& \partial_t T + \bs{u} \cdot \nabla T = \Delta T,
\end{align}
where both of the velocity fields $\bs{u}$ and $\bs{U}$ are divergence-free. Is it possible to produce example of smooth vector fields $\bs{U}$ for which $Ra^{-1} \langle |\nabla \bs{U}|^2 \rangle \sim Ra^{1/2}$ such that there is a solution $(\bs{u}, T)$ satisfying $Ra^{-1} \langle |\nabla \bs{u}|^2 \rangle = \langle u_2 T \rangle \gtrsim Ra^{1/2}$ as $Ra \to \infty$.
\end{question}

%\appendix

%\begin{center}
%{\bf APPENDIX}
%\end{center}

\bibliographystyle{apalike}
\bibliography{references.bib}

\end{document}